\documentclass[english, 12pt,leqno]{amsart}

\usepackage{a4,amsbsy,amsfonts,amsgen,amsmath,
amsopn,amssymb,amsthm,amstext,enumerate,enumitem,exscale,
stmaryrd,epsfig,a4,graphicx,mathrsfs,relsize,url}

\newtheorem{thm}{Theorem}
\newtheorem{cor}[thm]{Corollary} 
\newtheorem{lem}[thm]{Lemma}
\newtheorem{prop}[thm]{Proposition}

\theoremstyle{definition} 
\newtheorem{defn}[thm]{Definition}
\newtheorem{rem}[thm]{Remark}
\newtheorem{obs}[thm]{Observation}
\newtheorem{exe}[thm]{Example}

\theoremstyle{remark}

\newcommand{\N}{\mathbf N}
\newcommand{\Z}{\mathbf Z}
\newcommand{\C}{\mathbf C}

\newcommand{\Sy}{\textnormal{Sym}}
\newcommand{\Alt}{\textnormal{Alt}}
\newcommand{\su}{\textnormal{sup}}

\title[Conjugacy growth series]
{Conjugacy growth series of some infinitely generated groups}

\author{Roland Bacher}
\address{Institut Fourier,
Universit\'e de Grenoble,
100 rue des maths,
BP74,
\newline
38402 Saint-Martin d'H\`eres,
France}
\email{Roland.Bacher@ujf-grenoble.fr}

\author{Pierre de la Harpe}
\address{Section de math\'ematiques,
Universit\'e de Gen\`eve,
C.P.~64,
\newline
1211 Gen\`eve~4, 
Switzerland}
\email{Pierre.delaHarpe@unige.ch}

\thanks{The visit of the first Author in Geneva, where much of this work was done,
was supported by the Swiss National Science Foundation}

\subjclass[2000]{20F69, 20F65}

\keywords{Conjugacy growth series, finitary symmetric group, wreath product,
alternating group, partition function, Ramanujan congruences}

\date{June 15,  2016}

\begin{document}

\begin{abstract}
It is observed
that the conjugacy growth series of the infinite finitary symmetric group 
with respect to the generating set of transpositions 
is the generating series of the partition function.
Other conjugacy growth series are computed, 
for other generating sets,
for restricted permutational wreath products of finite groups by the finitary symmetric group,
and for alternating groups.
Similar methods are used to compute
usual growth polynomials and conjugacy growth polynomials
for finite symmetric groups and alternating groups, 
with respect to various generating sets of transpositions.
\par
Computations suggest a class of finite graphs,
that we call partition-complete,
which generalizes the class of semi-hamiltonian graphs,
and which is of independent interest.
\par
The coefficients of a series related to the finitary alternating group
satisfy congruence relations analogous to
Ramanujan congruences for the partition function.
They follow from partly conjectural ``generalized Ramanujan congruences'',
as we call them, for which we give numerical evidence in Appendix~\ref{AppendixC}.

\end{abstract}

\maketitle

\begin{itemize}
\item[]
\hskip3cm
\emph{Pour le parfait fl\^aneur, pour l'observateur passionn\'e,}
\item[]
\hskip3cm
\emph{c'est une immense jouissance 
que d'\'elire domicile}
\item[]
\hskip3cm
\emph{dans le nombre, dans l'ondoyant dans le mouvement,}
\item[]
\hskip3cm
\emph{dans le fugitif et l'infini.}
\item[]
\hskip3cm
(Baudelaire, in \emph{Le peintre de la vie moderne} \cite{Baud--63}.)
\end{itemize}

\section{\textbf{Explicit conjugation growth series}}
\label{explicit...series}

Let $G$ be a group generated by a set $S$.
For $g \in G$, the \textbf{word length} $\ell_{G,S} (g)$ 
is defined to be the smallest non-negative integer $n$
for which there are $s_1, s_2, \hdots, s_n \in S \cup S^{-1}$ 
such that $g=s_1s_2 \cdots s_n$,
and the \textbf{conjugacy length} $\kappa_{G,S}(g)$ is the smallest integer $n$
for which there exists $h$ in the conjugacy class of $g$
such that $\ell_{G,S}(h) = n$.
For $n \in \N$, denote by $\gamma_{G,S}(n) \in \N \cup \{\infty\}$
the number of conjugacy classes  in $G$ 
consisting of elements $g$ with $\kappa_{G,S}(g) = n$
(we agree that $0 \in \N$).
Assuming that the pair $(G,S)$ satisfies the condition
\begin{equation}
\tag{Fin}
\label{Fin}
\gamma_{G,S}(n) \hskip.2cm \text{is finite for all} \hskip.2cm n \in \N ,
\end{equation}
we define the \textbf{conjugacy growth series}
$$                  
C_{G,S}(q) \, = \, \sum_{n =0}^\infty \gamma_{G,S}(n) q^n 
\, = \, \sum_{g \in \text{Conj}(G)} q^{\kappa_{G,S}(g)} \, \in \, \N [[q]] \hskip.1cm .
$$
Here $\sum_{g \in \text{Conj}(G)}$ indicates a summation over a set of representatives in $G$
of the set of conjugacy classes of $G$.
The \textbf{exponential rate of conjugacy growth} is
$
H^{\text{conj}}_{G,S} \, = \, \limsup_{n \to \infty} \frac{ \log \gamma_{G,S}(n) }{n} \hskip.1cm ;
$
note that $\exp( -H^{\text{conj}}_{G,S})$ is the radius of convergence of the series $C_{G,S}(q)$.
\par

In case $G$ is generated by a finite set $S$, Condition (\ref{Fin})
is obviously satisfied, 
so that the formal series $C_{G, S}(q)$ and
the number $H^{\text{conj}}_{G,S}$
are well defined;
they have recently been given some attention, see e.g.\ 
\cite{AnCi}, \cite{BCLM--13}, \cite{Fink--14}, 
\cite{GuSa--10}, \cite{HuOs--13}, \cite[Chap.\ 17]{Mann--12}, \cite{PaPa--15}, \cite{Rivi--10}.
The subject is related to that of counting closed geodesics in compact Riemannian manifolds
\cite{Babe--88},  \cite{CoKn--04}, \cite{Hube--56}, \cite{Knie--83},  \cite{Marg--69}.
\par

When $S$ is finite, denote for $n \in \N$ by $\sigma_{G,S}(n) \in \N$
the number of elements $g \in G$ with $\ell_{G,S}(g) = n$.
In this situation, it is tempting to compare the series $C_{G,S}$ to the
\textbf{growth series}
$$
L_{G,S} (q) \, = \, \sum_{n=0}^\infty \sigma_{G,S} (n) q^n 
\, = \, \sum_{g \in G} q^{\ell_{G,S}(g)} \, \in \, \N [[ q ]] \hskip.1cm .
$$
\par
For finite series, e.g.\ for finite groups, we rather write ``conjugacy growth polynomial''
and ``growth polynomial''.

\vskip.2cm

The first purpose of the present article is to observe that 
there are groups $G$ which are not finitely generated, 
and yet have interesting series $C_{G,S}(q)$ for appropriate infinite generating sets $S$.
Groups of concern here are locally finite infinite symmetric groups, some of their wreath products, and infinite alternating groups.
We are also led to compute and compare polynomials $C_{G,S}$ and $L_{G,S}$
for finite symmetric and alternating groups, for various generating sets $S$.

\vskip.2cm

For a non-empty set $X$, we denote by $\Sy (X)$ the 
\textbf{finitary symmetric group} of $X$,
i.e.\ the group of permutations of $X$ with finite support.
The \textbf{support} of a permutation $g$ of $X$ is the subset
$\su (g) = \{ x \in X \vert g(x) \ne x \}$ of $X$.
Two permutations of $X$ are \textbf{disjoint} if their supports are disjoint
(below, this will be used mainly for cycles).
It is convenient to agree that, for $g,h \in \Sy (X)$, 
\begin{center}
we denote by $gh$
the result of the permutation $h$ \emph{followed by} $g$,
\end{center}
such that $(gh)(x) = g(h(x))$ for all $x \in X$.
For example, for $X = \N$, we have
$(1,2)(2,3) = (1,2,3)$, and not $(1,3,2)$ as  
with the other convention.
\par

The conjugacy class
$$
T_X \, = \,  \left\{ (x, y) \in \Sy(X) \mid
x,y \in X  \hskip.2cm \text{are distinct}
\right\}
 \, \subset \, \Sy (X) 
$$
 of all transpositions in $\Sy (X)$
is a generating set of $\Sy (X)$.
We consider also other generating sets, in particular for $X = \N$
$$
S^{\text{Cox}}_\N \, = \, \{ (i,i+1) \mid i \in \N \} ,
$$
which makes $\Sy (\N)$ look like an infinitely generated irreducible Coxeter group of type $A$.
\par

When $X$ is finite, $\Sy (X)$ is the usual symmetric group of $X$.
For $n \ge 1$ and $X =  \{1, 2, \hdots, n\}$, we write $\Sy (n)$.
The sets of transpositions
$$
S^{\text{Cox}}_n \, = \,  \{ (1,2), (2,3) , \cdots , (n-1,n) \} ,
\hskip.5cm
T_n \, = \,  \{ (i,j) \mid 1 \le i,j \le n,\hskip.1cm  i < j \} 
$$
are particular cases for $\Sy (n)$ of generating sets
which are standard for finite Coxeter groups.
\par

In the following proposition, 
we collect a sample of equalities
that appear again in Proposition \ref{sharpening}
in a more general situation.

\begin{prop}
\label{firstsample}
Let $S \subset \Sy (\N)$ be a generating set such that
$S^{\text{Cox}}_\N \subset S \subset T_\N$.
For every $n \ge 1$, let $S_n \subset \Sy (n)$ be a generating set such that
$S^{\text{Cox}}_n \subset S_n \subset T_n$. Then
\vskip.2cm
\begin{enumerate}[label=(\roman*)]
\setlength\itemsep{1em}
\item\label{iDEfirstsample}\hskip.5cm
$
C_{\Sy (\N), S}(q) 
   \, = \, \sum_{m=0}^\infty p(m) q^m
   \, = \,  \prod_{j=1}^\infty \frac{1}{1-q^j}  \hskip.1cm ,
$
\item[]
in particular the sequence of coefficients of $C_{\Sy (\N), S}(q)$
if of intermediate growth,
\item\label{iiDEfirstsample}\hskip.5cm
$
C_{\Sy (n), S_n}(q) \, = \, \sum_{k=0}^{n-1} p_{n-k}(n) q^k \hskip.1cm ,
$
\item\label{iiiDEfirstsample}\hskip.5cm
$
\sum_{n=0}^\infty C_{\Sy (n), S_n} (q) t^n \, = \, \prod_{j=1}^\infty \frac{1}{1 - q^{j-1} t^j} \hskip.1cm ,
$
\end{enumerate}
\vskip.2cm
where the partition function $p(n)$ and the second equality of \ref{iDEfirstsample}
are as recalled in Appendix \ref{recallp(n)},
and the number $p_{n-k}(n)$  of partitions of $n$ with $n-k$ positive parts
as in Appendix \ref{recallpk(n)}.
Moreover:
\begin{enumerate}[label=(\roman*)]
\addtocounter{enumi}{3}
\item\label{ivDEfirstsample}
when $n \to \infty$,
the polynomials $C_{\Sy (n), S_n}(q)$ of \ref{iiDEfirstsample}
converge coefficientwise towards the series $C_{\Sy (\N), S}(q)$
of \ref{iDEfirstsample}.
\end{enumerate}
\end{prop}
 
For example: 
$$
\aligned
C_{\Sy (2), S_2}(q) \, &= \,  1+q ,
\\
C_{\Sy (3), S_3}(q) \, &= \,   1 + q + q^2 ,
\\
C_{\Sy (4), S_4}(q) \, &= \,   1 + q +2 q^2 + q^3 , 
\\
C_{\Sy (5), S_5}(q) \, &= \,   1 + q + 2 q^2 + 2 q^3 + q^4 ,
\\
C_{\Sy (6), S_6}(q) \, &= \,   1 + q + 2 q^2 + 3 q^3 + 3 q^4 + q^5 \hskip.1cm ,
\\
C_{\Sy (n), S_n}(q) \, &= \, 1 + q  + 2q^2 + \cdots + \lfloor n/2 \rfloor q^{n-2} + q^{n-1}
\hskip.2cm (n \ge 5).
\endaligned
$$

The main ingredients for the proof of Proposition \ref{firstsample}
are the classical Observation \ref{partconjclasses} and Lemma \ref{lemmafcl}.
We use the following standard notation: for an integer $n \ge 0$, we denote by
$\lambda = (\lambda_1, \lambda_2, \cdots, \lambda_k) \vdash n$ a \emph{partition} 
of \emph{weight} $n = \lambda_1 + \lambda_2 + \cdots + \lambda_k$,
with $k \ge 0$ and $\lambda_1 \ge \lambda_2 \ge \cdots \ge \lambda_k \ge 1$.
\par
In Proposition \ref{betterthan1} of Section \ref{sectionproofs}, 
we come back to the convergence of 
$C_{\Sy (n), S_n}(q)$ to $C_{\Sy (\N), S}(q)$.

\begin{obs}
\label{partconjclasses}
Let $X$ be a non-empty set, finite or infinite.
Denote by  $\vert X \vert$ its cardinality. 
Conjugacy classes in $\Sy (X)$ are in natural bijection
with appropriate sets of partitions.
More precisely,
for each pair $(L,k)$ of non-negative integers with $L+k \le \vert X \vert$,
there is a bijection between the set of partitions of the form
\begin{equation}  
\tag{\ref{partconjclasses}.a}
\label{tag2.a}
\lambda \, = \,  (\lambda_1, \hdots, \lambda_k) \vdash L
\end{equation}
on the one hand,
and conjugacy classes in $\Sy (X)$ 
of elements of the form
\begin{equation}  
\tag{\ref{partconjclasses}.b}
\label{tag2.b}
\aligned
&g \, = \, c_1 c_2 \cdots c_k \in \Sy (X) \hskip.2cm \text{where}
\\
&\hskip.5cm c_i \hskip.2cm \text{is a cycle of some length $\lambda_i+1 \ge 2$ 
for $i = 1, 2, \hdots, k$,}
\\
&\hskip.5cm \text{$c_i$ and $c_{i'}$ are disjoint for $i \ne i'$,}
\\
&\hskip.5cm \text{and therefore} 
\hskip.2cm  \vert \su (g) \vert  - k = L  = \sum_{i=1}^k \lambda_i ,
\endaligned
\end{equation}
on the other hand.
In this article, 
\begin{center}
the length of a \emph{cycle} is at least $2$,
unless otherwise stated;
\end{center}
we always make it explicit when we want to consider \emph{fixed points}
as cycles of length $1$.
\end{obs}

\begin{lem}
\label{lemmafcl}
Consider two integers $L,k \ge 0$, 
a set $X$ of cardinal at least $L+k$ (possibly infinite),
an element $g \in \Sy (X)$ product of $k$ disjoint cycles with $\vert \su (g) \vert = L+k$, 
and the corresponding partition $\lambda \vdash L$ in $k$ parts,
as in Observation \ref{partconjclasses}.
\begin{enumerate}[noitemsep,label=(\roman*)]
\item\label{iDElemmafcl}
There exist transpositions $s_1, \hdots, s_L \in \Sy (X)$
such that $g = s_1 \cdots s_L$ and $\su (s_l) \subset \su (g)$
for all $l \in \{1, \hdots, L\}$.
\item\label{iiDElemmafcl}
There exist transpositions $t_1, \hdots, t_M \in \Sy (X)$ 
such that $g = t_1 \cdots t_M$
if and only if $M \ge L$ and $M-L$ is even.
\end{enumerate}
Suppose moreover that $X$ is given together with 
trees $T_1, \hdots, T_k$ with the following properties:
for $i \in \{1, \hdots, k\}$, the vertex set of $T_i$ is a subset of $X$
of cardinality $\lambda_i + 1$, 
and these subsets are disjoint from each other.
Let $\big\{ \{x_1, x'_1\}, \hdots, \{x_L, x'_L \} \big\}$
be an enumeration of the edges of the forest  $\bigcup_{i=1}^k T_i$.
\begin{enumerate}[noitemsep,label=(\roman*)]
\addtocounter{enumi}{2}
\item\label{iiiDElemmafcl}
The product $h = (x_1, x'_1) \cdots (x_L, x'_L)$ is conjugate
to $g$ in $\Sy (X)$.
\end{enumerate}
\end{lem}

We postpone until Section  \ref{sectionproofs}
the proofs of these, and of  further propositions in the present section.
Before we can state more general cases 
of some of the equalities of Proposition \ref{firstsample},
we introduce two definitions and provide examples.

\begin{defn}
\label{transpgraph}
For a set $S$ of transposition of a set $X$, 
the \textbf{transposition graph} $\Gamma (S)$ has
vertex set $X$ and 
edge set those pairs $\{x,y\} \subset X$ for which
the transposition $(x,y)$ is in $S$.
\end{defn}

But for their names, these graphs appear in \cite{Serg--93}.
It is well-known and easy to check (Lemma \ref{arbrecycle}) that
\begin{equation}
\tag{GC}
\label{GC}
\aligned
&\text{the group $\Sy (X)$ is generated by $S$}
\\
&\text{if and only if the graph $\Gamma (S)$ is connected} .
\endaligned
\end{equation}

\begin{defn}
\label{PCall}
For a set $X$, 
a set $S$ of transpositions of $X$
is \textbf{partition-complete} if it satisfies the following condition:
\begin{equation}
\tag{PC}
\label{PC}
\aligned
&\text{the transposition graph $\Gamma (S)$ is connected and, }
\\
&\text{for every partition $\lambda = (\lambda_1, \hdots, \lambda_k) \vdash L$
   such that $L+k \le \vert X \vert$,}
\\
&\text{$\Gamma (S)$ contains a forest consisting of $k$ trees}
\\
&\text{having respectively $\lambda_1+1, \hdots, \lambda_k+1$ vertices.}
\endaligned
\end{equation}
The graph $\Gamma (S)$ itself is partition-complete when $S$ is so.
\end{defn}

\begin{exe}
\label{examplesPC}
When $X = \{1, \hdots, n\}$,
sets of transpositions satisfying Condition (\ref{PC}) include 
sets $S$ such that
$S^{\text{Cox}}_n \subset S \subset T_n$, 
and also those for which $\Gamma (S)$
is one of the Dynkin graphs $\operatorname{D}_{2n+1}$ with $n \ge 2$, 
or  $\operatorname{E}_7$ or $\operatorname{E}_8$.
\par
But if $S$ is such that $\Gamma (S)$ 
is one of $\operatorname{D}_{2n}$ with $n \ge 2$, or $\operatorname{E}_6$,
then $S$ does not satisfy Condition (\ref{PC}),
because $\operatorname{D}_{2n}$ does not contain $n$ disjoint trees
with two vertices each,
and $\operatorname{E}_6$ does not contain two disjoint trees
with three vertices each.
\par
When $X$ is finite, 
$S$ is partition-complete as soon as 
the graph $\Gamma (S)$ is semi-hamiltonian;
recall that a graph is semi-hamiltonian [respectively hamilto\-nian]
if it contains a path [respectively a cycle]  containing every vertex exactly once.
Condition (\ref{PC}) for a graph
can be seen as a weakening of the property
of being semi-hamiltonian.
\par
When $X$ is infinite, Condition (\ref{PC}) is equivalent to
(\ref{PCinfinity}):
\begin{equation}
\tag{PC$_\infty$}
\label{PCinfinity}
\aligned
& \text{$S$ generates $\Sy (X)$ and, for all $n \ge 1$,}
\\
& \text{the graph $\Gamma (S)$ contains a disjoint union}
\\
& \text{of $n$ trees with at least $n$ vertices each.}
\endaligned
\end{equation}
\par

When $X = \N$, 
here are two families of examples of sets $S$ satisfying Condition (\ref{PCinfinity}).
The first is that of sets of transpositions
of which the transposition graph contains arbitrarily long segments;
this family contains sets $S$
such that $S^{\text{Cox}}_\N \subset S \subset T_\N$.
For a set of the second family, choose an  increasing sequence $(k_n)_{n \ge 1}$
of positive integers such that $k_{n+2}-k_{n+1} > k_{n+1}-k_n$ for all $n \ge 1$;
define then $S$ as the set of transpositions $(0, k_n)$ and $(k_n, j)$
for all $n \ge 1$ and $j$ with $k_n+1 \le j \le k_{n+1}-1$,
so that $\Gamma (S)$ is obtained from a star with centre $0$ 
and infinitely many neighbours $k_n$
by attaching $k_{n+1}-k_n-1$ vertices to each vertex $k_n$;
thus $\Gamma (S)$ is a tree of diameter $4$,
with all vertices but one (the origin) of finite degrees.
\par

On the contrary, the set
$S^0_\N = \{ (0, n) \mid n \ge 1\}$
does not satisfy Condition (\ref{PCinfinity}).
Proposition \ref{propS0} provides the conjugacy growth series
for the pair $(\Sy (\N), S^0_\N)$.

\par
\emph{We ignore the existence of
a simple criterion for graphs or trees to be partition complete.}
\end{exe}

Using Definition \ref{PCall}, 
we reformulate Lemma \ref{lemmafcl}\ref{iiiDElemmafcl}
and generalize Proposition \ref{firstsample} as follows:

\begin{lem}
\label{reformlemma3}
Let $X$ be a non-empty set and $S$ a partition-complete set of transpositions of $X$.
Let $g = c_1 \cdots c_k \in \Sy (X)$ be a product of disjoint cycles
of non-increasing lengths; 
denote these lengths by $\lambda_1+1, \hdots, \lambda_k+1$,
and set $L = \sum_{i=1}^k \lambda_i$, so that $\vert \su (g) \vert = L+k$.
Then
$$
\kappa_{\Sy (X), S}(g)  \, = \, L .
$$
\end{lem}

\begin{prop}
\label{sharpening}
Let $X$ be an infinite set and $S \subset \Sy (X)$ 
a partition-complete set of transpositions.
\begin{enumerate}[noitemsep]
\item[(a)]
The equalities of \ref{iDEfirstsample} 
in Proposition~\ref{firstsample} hold true.
In particular the series $C_{\Sy (X), S}(q)$ 
does not depend on the cardinality of $X$, as long as $X$ is infinite.
\end{enumerate}
For every $n \ge 1$, let $S_n \subset \Sy (n)$ be 
a partition-complete set of transpositions.
\begin{enumerate}[noitemsep]
\item[(b)]
Claims  \ref{iiDEfirstsample}, 
\ref{iiiDEfirstsample}, and \ref{ivDEfirstsample} 
in Proposition \ref{firstsample} hold true.
\end{enumerate}
\end{prop}

For the next proposition, we consider the generating sets of transpositions
$$
\aligned
S^0_\N \, &= \, \{ (0,n) \in \Sy (\N) \mid n \ge 1 \} \, \subset \, \Sy (\N) ,
\\
S^0_n \, &= \,  \{ (0,i) \mid 1 \le i \le n-1 \} \, \subset \,
\Sy (n) = \Sy (\{0, 1, \hdots, n-1\}) ,
\endaligned
$$
which  \emph{do not} satisfy Condition (\ref{PC}).

\begin{prop}
\label{propS0}
Let $S^0_\N \subset \Sy (\N)$ and, 
for every $n \ge 1$, let $S^0_n \subset \Sy (n)$
be as above. Then
$$
\aligned
(i) \hskip.2cm
&C_{\Sy (\N), S^0_\N}(q) \, = \, 
1 + \sum_{k=1}^\infty q^{3k-2} \prod_{j=1}^k \frac{1}{1-q^j} \, = \, 
\\
&\hskip.5cm
1 + q + q^2 + q^3 + 2q^4 + 2q^5 + 
3q^6 + 4q^7 + 5q^8 + 6q^9 + 9q^{10} +
\\
&\hskip.5cm
10 q^{11} + 13 q^{12} + 17q^{13} + 21 q^{14} + 25 q^{15} + 
33 q^{16} + 39 q^{17} + 49 q^{18} +
\\
&\hskip.5cm
60 q^{19} + 73 q^{20} + 88 q^{21} + 110 q^{22} + 130 q^{23} + 
158 q^{24} + \cdots \hskip.1cm ,
\\
(ii) \hskip.2cm
&C_{\Sy (n), S^0_n} (q) \, = \, 
1 \, + \, \sum_{k=1}^{\lfloor n/2 \rfloor} q^{2k-2} \sum_{j=k}^n p_k(j)q^{j}
\hskip.1cm .
\endaligned
$$
Moreover, when $n \to \infty$,
the polynomials $C_{\Sy (n), S^0_n}$ of (ii)
converge coefficientwise 
towards the series $C_{\Sy (\N), S^0_\N}(q)$ of (i).
\end{prop}

At the day of writing, the sequence
$$
\begin{array}{c|ccccccccccccccccccc}
n&0&1&2&3&4&5&6&7&8&9&10&11&12&13&14&15&16&17&18\\
\hline
c_n&1&1&1&1&2&2&3&4&5&6&9&10&13&17&21&25&33&39&49\\
\end{array}
$$
of coefficients of the series $C_{\Sy (\N), S^0_\N} (q) := \sum_{n=0}^\infty c_n q^n$
of (i) does not appear in \cite{OEIS}.
The equality of (ii) is repeated in Proposition \ref{growthpolSym(n)S0} below.
\par

Numerically, the series of Proposition \ref{propS0}(i)
converges in the unit disc,
and shows two roots of smallest absolute value,
near $-0.53\pm 0.68i$.
This makes it unlikely that the series of Proposition \ref{propS0}
has such a nice product expansion like that of Proposition \ref{firstsample}\ref{iDEfirstsample}.

\vskip.2cm

Let $X$ be an infinite set and $H$ a finite group.
Let $W = H \wr_X \Sy (X)$ be the corresponding permutational wreath product.
Let $S$ be a generating set of $W$ containing a set of transpositions $S_X$ of $X$
generating $\Sy (X)$ and satisfying Condition (\ref{PCwr}) of Section \ref{sectionwreath}.
Denote by $M$ the number of conjugacy classes of $H$.

\begin{prop}[see Proposition \ref{wreath} below]
\label{prewreath}
Let $W = H \wr_X \Sy (X)$, $S$ and $M$ be as above. Then
$$
C_{W,S} (q) \, = \, 
\Big( C_{\Sy (X), S_X} (q) \Big) ^N \, = \,
\prod_{k=1}^\infty \frac{1}{ (1-q^k)^M} \hskip.1cm .
$$
\end{prop}

The \textbf{finitary alternating group} of $\N$ is the subgroup $\Alt (\N)$ of $\Sy (\N)$
of permutations of even signature. Consider its generating set
$$
S^A_\N \, = \, \{ (i, i+1, i+2) \in \Alt (\N) \mid i \in \N \} ,
$$
as well as the subset $T^A_\N := \bigcup_{g \in \Alt (\N)} g S^A_\N g^{-1}$
of all $3$-cycles.
Proposition \ref{AltIntro} is the analogue for the finitary alternating group of $\N$
of Proposition \ref{firstsample}\ref{iDEfirstsample}
for the finitary symmetric group.

\begin{prop}   
\label{AltIntro}
Let $S \subset \Alt (\N)$ be a generating set 
such that $S^A_\N \subset S \subset T^A_N$.
Then
$$
\aligned
C_{\Alt (X), S} (q) \, &= \, 
\sum_{u=0}^\infty p(u) q^u \sum_{v=0}^\infty p_e(v) q^v
\\
&= \, \frac{1}{2} \prod_{j=1}^\infty \frac{1}{(1-q^j)^2} \, + \,
\frac{1}{2} \prod_{j=1}^\infty \frac{1}{1-q^{2j}}
\\
&= \, 1+q+3q^2+5q^3+11q^4+18q^5+34q^6
\\
& \hskip.5cm +55q^7+95q^8+150q^9+244q^{10} + \cdots \hskip.1cm .
\endaligned
$$
where $p_e(v)$ denotes the number of partitions of $v \in \N$
involving an even number of positive parts, as in Appendix \ref{recallpepo(n)}.
\end{prop}

\begin{obs}
\label{CongIntro}
For the series of Proposition \ref{AltIntro},
set 
$$
C_{\Alt (\N), S} (q) \, = \,  \sum_{n=0}^\infty p^A(n) q^n .
$$
The coefficients $p^A(n)$ satisfy the following congruence relations:
\begin{eqnarray*}
p^A(5n+3)&\equiv& 0\pmod 5,\\
p^A(10n+7)&\equiv& 0\pmod 5,\\
p^A(10n+9)&\equiv& 0\pmod 5, \\
p^A(25n+23)&\equiv& 0\pmod {25}.
\end{eqnarray*}
Moreover, \emph{conjecturally}:
\begin{eqnarray*}
p^A(49n+17)&\equiv& 0\pmod 7,\\
p^A(49n+31)&\equiv& 0\pmod 7,\\
p^A(49n+38)&\equiv& 0\pmod 7,\\
p^A(49n+45)&\equiv& 0\pmod 7,\\
p^A(121n+111)&\equiv &0\pmod {11}.
\end{eqnarray*}
See Proposition \ref{p^Asure} for the first four relations.
The conjectured relations have been verified numerically
for $p^A(m)$ when $m \le 5000$,
as discussed in Section \ref{sectionRama} and Appendix \ref{AppendixC}.
\end{obs}

\begin{rem}
\label{CompareSandT}
(i) 
Let $G$ be a group generated by a subset $T$.
Then 
\begin{equation}
\label{kappa=ell}
\tag{$\kappa_T=\ell_T$}
\kappa_{G,T}(g) \, = \, \ell_{G,T}(g)
\hskip.2cm \text{for all} \hskip.2cm g \in G 
\end{equation}
if and only if $T$ is closed by conjugation,
as it is straightforward to check.

\vskip.2cm

(ii)
Suppose that $G$ is also generated by a subset $S$, 
and assume that
$T = \bigcup_{h \in G} hSh^{-1}$. 
Then 
\begin{equation}
\label{kappakappa1}
\tag{$\kappa_T \le \kappa_S$}
\kappa_{G,T}(g) \, \le \, \kappa_{G,S}(g)
\hskip.2cm \text{for all} \hskip.2cm g \in G ,
\end{equation}
but equality need not hold.
\par

For example, if  $G = \Sy (4)$ and $S = \{(1,2), (2,3,4) \}$, then 
$$
\kappa_{G, T}((1,2)(3,4)) = 2 \, < \,  \kappa_{G,S}((1,2)(3,4)) = 4 .
$$

\vskip.2cm

(iii)
It is remarkable that we have
\begin{equation}
\label{kappakappa2}
\tag{$\kappa_T = \kappa_S$}
\kappa_{G,T}(g) \, = \, \kappa_{G,S}(g)
\hskip.2cm \text{for all} \hskip.2cm g \in G ,
\end{equation}
in many cases of interest here, including
\begin{itemize}
\item[--]
$G = \Sy (\N)$ and $S$ as in Proposition \ref{firstsample}\ref{iDEfirstsample},
so that $T = T_\N$,
\item[--]
$G = \Sy (n)$ and $S = S_n$ 
as in Proposition \ref{firstsample}\ref{iiDEfirstsample}, 
so that $T = T_n$,
\item[--]
$G = \Alt (\N)$ and $S$ as in Proposition \ref{AltIntro},
so that $T = T^A_\N$.
\end{itemize}
In these cases, it follows that 
\begin{equation}
\label{CC2}
\tag{$C_T = C_S$}
C_{G, T}(q) \, = \,  C_{G, S}(q) .
\end{equation}
\par
Note however that, in the case of
$G = \Sy (\N)$ and $S = S^0_\N$, 
and therefore $T = T_\N$, 
the series $C_{G, S}(q)$ of Proposition \ref{propS0}
and $C_{G, T}(q)$ of Proposition \ref{firstsample}\ref{iDEfirstsample} 
are different,
so that the equalities (\ref{kappakappa2}) and (\ref{CC2}) do not hold.
\end{rem}

\subsection*{Overview}
Section \ref{sectionproofs} contains proofs of 
Propositions \ref{firstsample},
\ref{sharpening},
\ref{propS0}
and Lemmas \ref{lemmafcl},
\ref{reformlemma3}.
In Section \ref{sectionwreath}, we write and prove
formulas for conjugacy growth series of wreath products,
see Propositions  \ref{prewreath} and \ref{wreath}.

Suppose that $G$ is a finite symmetric group $\Sy (n)$,
and $S$ a system of generators.
When $S$ is either $S^{\text{Cox}}_n$ or $T_n$,
the polynomial $L_{G, S}(q)$ is well-known,
and is recalled in Proposition \ref{rodrigues} below.
Indeed, these polynomials make sense and are explicitely known
for all finite Coxeter systems;
they appear in many places,
for example \cite{Solo--66} and \cite[exercises of $\S$ IV.1]{Bour--68},
as well as  \cite{ShTo--54}.
In Section  \ref{sectionLCsym}, we compute $C_{\Sy (n),S}(q)$,
and compare these polynomials
with those for another generating set, the set $S^0_n$ defined above;
this uses lemmas of Section~\ref{sectionproofs},
as well as some facts on derangements
recalled in Appendix \ref{AppendixB}.
\par

In  Section \ref{sectionalt}, we present results 
of analogous computations for finitary alternating groups,
and in particular the proof of Proposition \ref{AltIntro}.
In the final Section \ref{sectionRama},
we discuss the context of Observation \ref{CongIntro}.
\par

There is a short Appendix \ref{AppendixA} with three lemmas
on symmetric and alternating groups,
and a longer Appendix \ref{AppendixB} 
that is a reminder of various definitions and identities
involving partitions and derangements.
Finally, in Appendix \ref{AppendixC}, we define a generalization of Ramanujan congruences
and we record a large number of these
for the coefficients $p(n)_{(e_1, e_2, e_3, \hdots)}$ of the power series
$$
\sum_{n=0}^\infty p(n)_{(e_1, e_2, e_3, \hdots)} q^n \, = \
\prod_{n=1}^\infty \frac{1}{ (1-q^n)^{e_1} (1-q^{2n})^{e_2} (1 - q^{3n})^{e_3} \, \cdots } \hskip.1cm ,
$$
where $(e_1, e_2, e_3, \hdots)$ is a finite sequence of non-negative integers.
Some of these congruences are established in the literature,
but most are (as far as we know) conjectural only, based on our numerical evidence.

\section{\textbf{Proof of Lemma \ref{lemmafcl} and \ref{reformlemma3}, 
and Propositions \ref{firstsample},
\ref{sharpening}, and  \ref{propS0} }}
\label{sectionproofs}

We will moreover state and prove a sharpening of Proposition
\ref{firstsample}\ref{ivDEfirstsample}, in Proposition \ref{betterthan1}.

\subsection{Proof of Lemmas \ref{lemmafcl} and \ref{reformlemma3}}
\label{subsection2.a}

As a preliminary step for the proof,
consider a cycle 
$$
c  \, = \, (x_1, \hdots, x_{\mu+1}) \in \Sy (X) ,
$$
where $1 \le \mu \le \vert X \vert - 1$.
By Lemma \ref{twocycles} applied $\mu - 1$ times (see Appendix \ref{AppendixA}),
the cycle $c$ can be written
as a product of $\mu$ transpositions with supports in $\su (c)$.

\vskip.2cm

Let $g \in \Sy (X)$ and $\lambda = (\lambda_1, \hdots, \lambda_k) \vdash L+k$ 
be as in Lemma \ref{lemmafcl}.
Write $g = c_1 \cdots c_k$, where $c_1, \hdots, c_k$ are disjoint cycles
of lengths $\lambda_1 + 1, \hdots, \lambda_k + 1$ respectively.
For $i \in \{1, \hdots, k\}$, it follows from the preliminary step
that $c_i$ can be written as a product of $\lambda_i$
transpositions with supports in $\su (g)$.
Hence $g$ can be written 
as a product of $L = \sum_{i=1}^k \lambda_i$ transpositions
with supports in $\su (g)$.
This proves \ref{iDElemmafcl} of Lemma \ref{lemmafcl}.

\vskip.2cm

With the extra ingredient of Lemma \ref{arbrecycle},
this also proves \ref{iiiDElemmafcl} of Lemma \ref{lemmafcl}
and Lemma \ref{reformlemma3}.

\vskip.2cm

Consider now $g = t_1 \cdots t_M$ as in 
\ref{iiDElemmafcl} of Lemma \ref{lemmafcl}.
For $i = 1, \hdots, k$, write
$c_i = (x^i_1, x^i_2, \hdots, x^i_{\lambda_i + 1})$.
Define a multigraph $G = G(t_1, \hdots, t_M)$ as follows:
its vertex set is $V_G := \bigcup_{\nu = 1}^M \su (t_\nu)$,
and there is one edge between the two vertices of $\su (t_\nu)$
for each $\nu \in \{1, \hdots, M\}$.
Observe that $V_G \supset \su (g) = \bigcup_{i=1}^k \su (c_i)$.
\par

Erasing from the product $t_1 \cdots t_M$
those $t_\nu$ contributing to connected components of $G$
disjoint from $\su (g)$ does not change this product.
We can therefore assume that each connected component of $G$
intersects $\su (g)$.
For each $i \in \{1, \hdots, k\}$ and $j \in \{1, \hdots, \lambda_i + 1\}$,
the connected component of $G$ containing $x^i_j$ contains $\su (c_i)$;
it follows that each connected component of $G$ 
contains at least one of the $\su (c_i)$~'s, 
and therefore that the number of connected components of $G$, 
say $\gamma_G$, is at most $k$.
\par

Given any finite multigraph with 
$v$ vertices, $e$ edges, and $\gamma$ connected components,
$e \ge v-\gamma$, with equality if and only if the multigraph is a forest.
For the multigraph $G$, we have therefore
$$
M \, \ge \, \vert V_G \vert - \gamma_G \, \ge \, \vert \su (g) \vert - k
\, = \, \sum_{i=1}^k \lambda_i \hskip.1cm .
$$
Moreover, $M$ and $L$ have the same parity,
which is also the signature of $g$.
\par

Conversely, for every $M \ge L$ with $M-L$ even,
$g$ can be written as a product of $M$ transpositions,
for example the $L$ transpositions of \ref{iDElemmafcl}
and $(M-L)/2$ times the product $s_1s_1$.
This proves \ref{iiDElemmafcl} of Lemma \ref{lemmafcl}.
\hfill $\square$

\subsection{Proof of  Propositions \ref{firstsample} and \ref{sharpening}}
\label{subsection2.b}
We prove the equalities of Proposition \ref{firstsample}
in the more general case of  Proposition~\ref{sharpening}.

\vskip.2cm 

\ref{iDEfirstsample}
Let $X$ be an infinite set and $S \subset \Sy (X)$ 
a partition-complete set of transpositions.
The series $C_{\Sy (X), S}(q)$ is a sum over partitions $\lambda \vdash L$
as in (\ref{tag2.a}) of  Observation \ref{partconjclasses}, 
and the contribution of such a partition is $q^L$
by Lemma \ref{reformlemma3}. Hence
$C_{\Sy (X), S}(q) = \sum_{L=0}^\infty p(L) q^L$.
Equality with $\prod_{j=1}^\infty \frac{1}{1-q^j}$
is Euler's identity (\ref{EulerProd})  recalled in Appendix \ref{recallp(n)}.

\vskip.2cm

\ref{iiDEfirstsample}
Consider a positive integer $n$ and a partition-complete set $S_n \subset \Sy (n)$.
Conjugacy classes in $\Sy (n)$ are now in bijection
with partitions of $n$ as follows:
a partition $(\mu_1, \hdots, \mu_k) \vdash n$
with exactly $k$ positive parts
corresponds to a permutation 
$g = c_1 \cdots c_k$ where $c_j$ is a cycle of length $\mu_j$,
and ``cycles'' of length $1$, i.e.\ fixed points of $g$, are now allowed
(this is why we use $\mu$ here rather than $\lambda$ as above).
By Lemma \ref{reformlemma3},
the $S_n$-conjugacy length of such a $g$ is
$\kappa_{\Sy (n), S_n}(g) = \sum_{j=1}^k (\mu_j - 1) = n-k$.
Hence the polynomial $C_{\Sy (n), S_n}(q)$
is a sum over partitions of $n$ (where $n$ is fixed) 
with exactly $k$ parts 
(where $k$ ranges from $1$ (long cycles) to $n$ (identity)),
and each such partition contributes by $q^{n-k}$.
Hence $C_{\Sy (n), S_n}(q) = \sum_{k=1}^n p_k(n)q^{n-k}
= \sum_{k=0}^{n-1} p_{n-k}(n)q^k$.

\vskip.2cm

\ref{iiiDEfirstsample}
Exchanging product and sum, we have
$$
\aligned
\prod_{k=1}^\infty \frac{1}{1 - q^{k-1}t^k} 
\, &= \, 
\prod_{k=1}^\infty \sum_{\ell_k = 0}^\infty q^{\ell_k (k-1)} t^{\ell_k k} 
\, = \, 
\sum_{\ell_1, \ell_1, \ell_3, \hdots \ge 0} \prod_{k=1}^\infty q^{\ell_k (k-1)} t^{\ell_k k}
\\
\, &= \, 
\sum_{n = 0}^\infty \Bigg( 
   \sum_{\substack{\ell_1, \ell_2, \ell_3, \hdots \ge 0 \\ 
   \ell_1 + 2\ell_2 + \cdots + k\ell_k + \cdots = n}} 
q^{\sum_{k=1}^\infty \ell_k (k-1)} \Bigg)  
t^n .
\endaligned
$$
For $n \ge 0$, there is a contribution to the coefficient of $t^n$
for each sequence $(\ell_1, \ell_2, \ell_3, \hdots)$ of non-negative integers
such that $\ell_1 + 2\ell_2 + 3\ell_3 + \cdots = n$,
equivalently for each partition $1^{\ell_1} 2^{\ell_2} 3^{\ell_3} \cdots$ of $n$,
with $\ell_1$ parts $1$, and $\ell_2$ parts $2$, and $\ell_3$ parts $3$, ..., 
equivalently for each conjugacy class in $\Sy (n)$.
Since $0\ell_1 + 1\ell_2 + 2\ell_3 + \cdots$ is the $S_n$-length of such a conjugacy class,
the contributions to the coefficient of $t^n$
add up precisely to $C_{\Sy (n), S_n} (q)$.

\vskip.2cm

\ref{ivDEfirstsample}
The polynomials of \ref{iiDEfirstsample} converge coefficientwise
towards the series of \ref{iDEfirstsample} 
because  $p_{n-k}(n)=p(k)$ when $2k \le n$.
See (\ref{Euler??}) in Appendix \ref{recallpk(n)}).
\hfill $\square$

\subsection{A computation of lengths}
For the next two lemmas, 
we agree that $\Sy (n)$ denotes the group of permutations
of $\{0, 1, \hdots, n-1\}$, and we consider the generating set
$S^0_n$ defined just before Proposition \ref{propS0}.

\begin{lem}
\label{longS0smaller}
Let $g = c_1 c_2 \hdots c_k \in \Sy (n)$, where $c_1, \hdots, c_k$ are disjoint cycles,
each of length at least $2$;
set $m = \vert \su (g) \vert$.
$$
\ell_{\Sy (n), S^0_n} (g) \, \le \, 
\left\{ \aligned
m+k \hskip.8cm &\text{if} \hskip.2cm g(0) = 0,
\\
m+k-2 \hskip.5cm &\text{if} \hskip.2cm g(0) \ne 0 .
\endaligned \right.
$$
\end{lem}

\begin{proof}
Choose $i \in \{1, \hdots, k\}$.
Let $\mu_i$ denote the length of $c_i$, and write
$c_i = (x_1, x_2, \hdots, x_{\mu_i})$.
\par

If $\su (c_i)$ does not contain $0$, then 
$$
c_i \, = \, (0,x_1) (0, x_{\mu_i}) (0, x_{\mu_i - 1}) \cdots (0,x_2) (0, x_1)
$$
and $\ell_{\Sy (n), S^0_n}(c_i) \le \mu_i+1$.
If $\su (c_i)$ contains $0$, say $x_1 = 0$,
(this occurs for at most one value of $i$),  then 
$$
c_i \, = \, (0, x_{\mu_i}) (0, x_{\mu_i - 1}) (0, x_{\mu_i - 2})  \cdots (0, x_2) 
$$
and $\ell_{\Sy (n), S^0_n}(c_i) \le \mu_i-1$.
\par

Since $\ell_{\Sy (n), S^0_n}(g)\le \sum_{i=1}^k\ell_{\Sy (n), S^0_n}(c_i)$, 
the lemma follows.
\end{proof}

\begin{lem}
\label{longS0equalfinite}
Let $g = c_1 c_2 \hdots c_k \in \Sy (n)$ and $m = \vert \su (g) \vert$
be as in the previous lemma.
Then
$$
\aligned
\ell_{\Sy (n), S^0_n} (g) \, &= \, 
\left\{ \aligned
m+k \hskip.8cm &\text{if} \hskip.2cm g(0) = 0,
\\
m+k-2 \hskip.5cm &\text{if} \hskip.2cm g(0) \ne 0 ,
\endaligned \right.
\\
\kappa_{\Sy (n), S^0_n} (g) \, &= \, 
\phantom{bb}
m+k-2  \hskip.5cm \text{as soon as $g \ne \operatorname{id}$.}
\endaligned
$$
\end{lem}

\begin{proof}
Set $L  =  \ell_{\Sy (n), S^0_n} (g)$;
there exist $r_1, \hdots, r_L \in S^0_n$
such that $g = r_1 r_2 \cdots r_L$.
For $i \in \{1, \hdots, k\}$, there are distinct elements
$x^i_1, \hdots, x^i_{\mu_i} \in \{0, 1, \hdots, n-1\}$
such that $c_i = (x^i_1, x^i_2, \hdots, x^i_{\mu_i})$;
and $\mu_1 + \cdots + \mu_k = m$.
Observe that, for all $i \in \{1, \hdots, k\}$
and $j \in \{1, \hdots, \mu_i\}$,
the transposition $(0, x^i_j)$ occurs in the list $r_1, \hdots, r_L$,
at least once.
\par

Suppose first that $0 \notin \su (g)$.
We know from Lemma \ref{longS0smaller} that $L \le m+k$.
If one had $L < m+k$,
there would exist $i \in \{1, \hdots, k\}$ 
such that $(0,x)$ occurs only one time in the list $r_1, \hdots, r_L$
for each $x \in \su (c_i)$;
but this is not possible since $0 \notin \su (c_i)$. Hence $L = m+k$.
\par

Suppose now that $0 \in \su (g)$;
we can assume that $x^1_1 = 0$.
We know from Lemma \ref{longS0smaller} that $L \le m+k-2$.
If one had $L < m+k-2$,
at least one of the two following situations would hold:
\begin{enumerate}[noitemsep,label=(\alph*)]
\item\label{IDElongS0equalfinite}
there exists $i \in \{2, \hdots, k\}$ 
such that $(0,x)$ occurs only one time in the list  $r_1, \hdots, r_L$
for each $x \in \su (c_i)$,
\item\label{IIDElongS0equalfinite}
there exists
$j \in \{2, 3, \hdots, \mu_1\}$
such that the transposition $(0, x^1_j)$ does not occur
in the list $r_1, \hdots, r_L$;
\end{enumerate}
but this is not possible. Hence $L = m+k-2$, and the formula for $\ell_{\Sy (n), S^0_n} (g)$ follows.
\par
For all $g \ne \operatorname{id}$ in $\Sy (n)$, 
there exists a conjugate $h$ of $g$ such that $h(0) \ne 0$
to which the same computation applies.
The formula for $\kappa_{\Sy (n), S^0_n}(g)$ follows.
\end{proof}

Similarly:

\begin{lem}
\label{longS0equalinf}
Let $g = c_1 c_2 \cdots c_k \in \Sy (\N)$, 
where $c_1, \hdots, c_k$ are disjoint cycles, each of length at least $2$;
set $m = \vert \su (g) \vert$.
Then $\ell_{\Sy (\N), S^0_\N}(g)$ and $\kappa_{\Sy (\N), S^0_\N}(g)$ are given by 
the formulas of the previous lemma.
\end{lem}

\subsection{Proof of Proposition \ref{propS0}}
\label{subsectionproofS0}
We record a minor variation of  Observation \ref{partconjclasses}, as follows.
Given $m \ge 2$ and $k \ge 1$,
there is a bijection between
\begin{enumerate}[noitemsep,label=(\alph*)]
\item\label{aDEproofS0}
the set of partitions of $m$ with $k$ parts, all at least $2$,
\\
(i.e.\ partitions of the form
$\mu = (\mu _1, \hdots, \mu _k) \vdash m$
with $\mu _1 \ge \cdots \ge \mu _k \ge 2$),
\end{enumerate}
and
\begin{enumerate}[noitemsep,label=(\alph*)]
\addtocounter{enumi}{1}
\item\label{bDEproofS0}
the set of conjugacy classes of elements $g \ne 1$ 
in $\Sy (\N)$ or $\Sy (n)$,
\\
with $\vert \su (g) \vert = m$,
which are products of $k$ disjoint cycles,
\\
where moreover $m \le n$ in the case of $\Sy (n)$
\\
(i.e.\ of elements of the form $g = c_1 \cdots c_k$ 
with $\operatorname{length}(c_i) = \mu _i$).
\end{enumerate}
For each $\mu$ as in \ref{aDEproofS0}, set
\begin{enumerate}[noitemsep,label=(\alph*)]
\addtocounter{enumi}{2}
\item[]
$\nu = (\nu_1, \hdots, \nu_k) := (\mu_1 - 1, \cdots, \mu_k - 1)
\vdash m-k$.
\\
which is a partition in $k$ positive parts.
\end{enumerate}
The relevant length of the conjugacy class of $g$ 
as in \ref{bDEproofS0} is $m+k-2$,
by Lemmas \ref{longS0equalfinite} and \ref{longS0equalinf}.

\vskip.2cm

For (i) of Proposition \ref{propS0}, it follows that 
$$
\aligned
C_{\Sy (\N), S^0_\N}(q)
\, &= \, 
\sum_{m=0}^\infty \gamma_{\Sy(\N), S^0_\N} (m) q^m 
\\
\, &= \, 
1 + \sum_{m=2}^\infty  \sum_{k=1}^{\lfloor m/2 \rfloor}
      p_k(m-k) q^{m+k-2}
\\
\, &= \, 
1 + \sum_{k=1}^\infty  q^{2k-2} \sum_{m=2k}^\infty p_k (m-k) q^{m-k}
\\
\, &= \, 
1 + \sum_{k=1}^\infty  q^{2k-2} \sum_{n=k}^\infty p_k (n) q^{n}
\\
\, &= \, 
1 + \sum_{k=1}^\infty q^{3k-2} \prod_{j=1}^k \frac{1}{1-q^j}
\endaligned
$$
where the last equality holds by (\ref{Euler312}) of Appendix \ref{recallpk(n)}.

\vskip.2cm

(ii)
Similarly:
$$
\aligned
C_{\Sy (n), S^0_n} (q) \, &= \, 
1  \, + \,  \sum_{m=2}^n  \sum_{k=1}^{\lfloor m/2 \rfloor}
p_k (m-k) q^{m+k-2}
\\
\, &= \, 
1 \, + \, \sum_{k=1}^{\lfloor n/2 \rfloor} q^{2k-2} \sum_{m=2k}^n p_k(m-k)q^{m-k}
\\
\, &= \, 
1 \, + \, \sum_{k=1}^{\lfloor n/2 \rfloor} q^{2k-2} \sum_{j=k}^n p_k(j)q^{j}
\hskip.1cm .
\endaligned
$$
(\emph{Note:}
$\sum_{j=k}^n p_k(j)q^{j} =  \sum_{j=0}^n p_k(j)q^{j}$.)
It is now clear that these polynomials converge coefficientwise to
$1 + \sum_{m=2}^\infty  \sum_{k=1}^{\lfloor m/2 \rfloor} p_k(m-k) q^{m+k-2}$,
that is to $C_{\Sy (\N), S^0_\N}(q)$.
 \hfill $\square$
 
\vskip.2cm

    We end this section with a sharpening of 
Claim \ref{ivDEfirstsample} of Proposition \ref{firstsample};
this applies more generally to the situation of Proposition \ref{sharpening}.
Let $S \subset \Sy (\N)$ be a partition-complete set of transpositions,
and let $L$ be a non-negative integer.
Set
$$
\mathcal K_L(S) \, = \, \{ g \in \Sy (\N) \mid \kappa_{\Sy (\N), S}(g) = L \} .
$$
Observe that $\mathcal K_L(S)$ is a union of conjugacy classes in $\Sy (\N)$.
For $g \in \Sy (\N)$, we denote by $k_g$ the number of disjoint cycles
of which $g$ is the product.

\begin{lem}
\label{lemmalimit}
Let $S$, $L$, and $\mathcal K_L(S)$ be as above.
\begin{enumerate}[noitemsep,label=(\roman*)]
\item\label{iDElemmalimit}
Let $g \in \mathcal K_L(S)$. Then
$\vert \su (g) \vert = L + k_g \le 2L$
for all $g \in \mathcal K_L(S)$.
Equality $k_g = L$ holds if and only if $g$ is a product
of $L$ disjoint transpositions.
\item\label{iiDElemmalimit}
Let $s \in \N$ be such that $0 \le s \le L/2$.
Then $\mathcal K_L(S)$ contains exactly $p(s)$ conjugacy classes
of elements $g$ such that $\vert \su (g) \vert = 2L-s$.
\end{enumerate}
\end{lem}

\begin{proof}
\ref{iDElemmalimit}
Let $g \in \mathcal K_L(S)$ be written as a product
$c_1 \cdots c_{k_g}$ of disjoint cycles of decreasing sizes.
For $i \in \{1, \hdots, k_g\}$, 
denote by $\lambda_i + 1$ the length of $c_i$;
set $\lambda = (\lambda_1, \hdots, \lambda_{k_g})$,
so that $\lambda \vdash L$ by Lemma \ref{reformlemma3}.
Since $k_g \le L$,
we have $\vert \su (g) \vert = L + k_g \le 2L$.
If $\vert \su (g) \vert = 2L$, then $\lambda_i = 1$ for $i=1, \hdots, k_g$,
and every $c_i$ is a transposition.

\vskip.2cm

\ref{iiDElemmalimit}
Let $s$ be such that $0 \le s \le L/2$.
We proceed to establish a bijection
between the set of partitions of $s$ on the one hand,
and the set of conjugacy classes of elements $g \in \Sy (\N)$
such that $g \in \mathcal K_L(S)$ and $\vert \su (g) \vert = 2L-s$
on the other hand; this will end the proof.
As Claim \ref{iDElemmalimit} covers the case $s=0$,
we could assume that $s \ge 1$.
\par

Choose a partition $\mu = (\mu_1, \hdots, \mu_m) \vdash s$.
Since $s \le L/2$, we have $L-s \ge m$. Set
$$
\lambda \, = \, (\lambda_1, \hdots, \lambda_{L-s}) \, = \, 
(\mu_1 + 1, \hdots, \mu_m + 1, 1, \hdots 1) ,
$$
a partition of $L$ with $L-(s+m)$ parts $1$. 
Let $g \in \Sy (\N)$ be a product of disjoint cycles of lengths
$\lambda_1 + 1, \hdots, \lambda_{L-s} + 1$.
Then
$$
\kappa_{\Sy (\N), S}(g) \, = \, \sum_{j=1}^{L-s} \lambda_j
\, = \, \Big( \sum_{j=1}^m \mu_j \Big) + L-s \, = \, L ,
$$
in particular $g \in \mathcal K_L(S)$, and
$$
\vert \su (g) \vert \, = \, \sum_{j=1}^{L-s} (\lambda_j + 1) \, = \, 2L-s .
$$
\par
Conversely, choose $g \in \mathcal K_L(S)$ with $\vert \su (g) \vert = 2L-s$.
Let $\lambda_1 + 1, \hdots, \lambda_{L-s} + 1$
be the lengths, in decreasing order, of the disjoint cycles
of which $g$ is the product; 
note that $\lambda = (\lambda_1, \hdots, \lambda_{L-s}) \vdash L$.
Define a partition $\mu = (\mu_1, \hdots, \mu_m)$
by $m = \max \{ j \in \{1, \hdots, L-s\} \hskip.1cm \vert \hskip.1cm \lambda_j \ge 2 \}$,
and $\mu_j = \lambda_j - 1$ for $j \in \{1, \hdots, m\}$. 
Then $\mu \vdash L - (L-s) = s$.
\end{proof}

Here is the announced sharpening, see Propositions \ref{firstsample} and \ref{sharpening}.

\begin{prop}
\label{betterthan1}
Let $S$ be a partition-complete set of transpositions in $\Sy (\N)$
and, for each $m \ge 1$, let $S_m$ be a partition-complete set of transpositions in $\Sy (m)$.
Write $C_\infty(q)$ for $C_{\Sy(\N), S}(q)$
and $C_m(q)$ for $C_{\Sy(m), S_m}(q)$.
Then:
$$
\aligned
&
\lim_{n \to \infty} \frac{1}{q^{n+1}} 
   \big( C_\infty(q) - C_{2n+1}(q) \big)
\, = \, 
\sum_{i=0}^\infty p(\le 2i)q^i \hskip.1cm ,
\\
&
\lim_{n \to \infty} \frac{1}{q^{n+1}} 
   \big( C_\infty(q) - C_{2n}(q) \big)
\, = \, 
\sum_{i=0}^\infty p(\le (2i+1))q^i \hskip.1cm ,
\endaligned
$$
where $p(\le j) := p(0)+p(1)+\dots+p(j)$ for all $j \in \N$.
\end{prop}

\begin{proof}
Note first that, for $L, m, k \in \N$, 
a conjugacy class in $\mathcal K_L(S)$
of elements $g$ such that $\vert \su (g) \vert = L+k$
intersects $\Sy (m)$ if and only if $L+k \le m$.
\par

Let $n \ge 1$. Choose an integer $k$ such that $1 \le k \le \frac{n+4}{3}$.
Let $\mathcal C$ be a conjugacy class in $\Sy (\N)$
such that $\mathcal C \subset \mathcal K_{n+k}(S)$.
\par

Suppose that $\mathcal C$ contributes to the coefficient of $q^{n+k}$ in $C_\infty (q)$
and not to the coefficient of $q^{n+k}$ in $C_{2n+1}(q)$.
Equivalently, suppose that, for every $g \in \mathcal C$,
we have $\vert \su (g) \vert \ge 2n+2$;
if $s \ge 0$ is defined by $\vert \su (g) \vert = 2(n+k)-s$,
this means that $s \le 2k-2$.
Since $k \le \frac{n+4}{3}$, i.e.\ $\frac{3k-4}{2} \le \frac{n}{2}$,
we have $s \le \frac{3k-4}{2} + \frac{k}{2} \le \frac{n+k}{2}$,
so that $\mathcal C$ is one of the $\sum_{s=0}^{2k-2} p(s)$ classes 
which appear in Lemma \ref{lemmalimit}\ref{iiDElemmalimit}.
It follows that the coefficient of $q^{n+k}$
in $C_\infty (q) - C_{2n+1}(q)$ is $p(\le (2k-2))$,
so that the coefficient of $q^{k-1}$ in $\frac{1}{q^{n+1}} (C_\infty (q) - C_{2n+1}(q))$
is $p(\le(2k-2))$ for $k$ with $1 \le k \le \frac{n+4}{3}$.
Consequently, for given $i \in \N$, the coefficient of $q^i$ 
in $\frac{1}{q^{n+1}} (C_\infty (q) - C_{2n+1}(q))$ is $p(\le 2i)$
as soon as $n$ is large enough.
\par

Similarly, suppose that $\mathcal C$ contributes to the coefficient of $q^{n+k}$ in $C_\infty (q)$
and not to the coefficient of $q^{n+k}$ in $C_{2n}(q)$.
A similar argument shows that $\mathcal C$
is one of the $\sum_{s=0}^{2k-1} p(s)$ classes 
which appear in Lemma \ref{lemmalimit}\ref{iiDElemmalimit},
and finally that, for $i \in \N$,  the coefficients of $q^i$
in $\frac{1}{q^{n+1}} (C_\infty (q) - C_{2n}(q))$
is $p(\le (2i+1))$ for $n$ large enough.
\end{proof}

\section{\textbf{Some wreath products}}
\label{sectionwreath}
Consider a non-empty set $X$, a group $H$,
and the \textbf{permutational wreath product} $H \wr_X \Sy (X) := H^{(X)} \rtimes \Sy (X)$.
Here, $H^{(X)}$ denotes the group of functions from $X$ to $H$
having finite support, for the pointwise multiplication,
and the semi-direct product ``$\rtimes$'' refers to 
the natural action of $\Sy (X)$ on $H^{(X)}$,
i.e.\ to $f \in \Sy (X)$ acting 
on $\psi \in H^{(X)}$
by $\psi \longmapsto f(\psi) := \psi \circ f^{-1}$.
The multiplication in this wreath product is given by
$(\varphi, f)(\psi, g) = (\varphi f(\psi), fg)$,
for $\varphi, \psi \in H^{(X)}$ and $f,g \in \Sy (X)$.
There is a natural action of the group $H \wr_X \Sy (X)$ on the set $H \times X$,
for which $(\varphi, f)$ acts by $(h,x) \longmapsto (\varphi(f(x))h, f(x))$;
this action is faithful.
\par

For $a \in H \smallsetminus \{1\}$ and  $u \in X$, 
denote by $\varphi^a_u \in H \wr_X \Sy (X)$
the permutation that maps $(h,x) \in H \times X$ to
$(ah, u)$ if $x=u$, and to $(h,x)$ otherwise; 
the support of $\varphi^a_u$ is the set $\{(h,u)\}_{h \in H}$.
Observe that $\left( \varphi^a_u \right)_{a \in H \smallsetminus \{1\}, u \in X}$ generates
the subgroup $H^{(X)}$, 
and that $\varphi^a_u, \varphi^b_v$ are conjugate in $H \wr_X \Sy (X)$
if and only if $a, b$ are conjugate in $H$.
\par

For $u \in X$, we denote by $H_u$ the set of elements 
$\varphi^a_u$ for $a \in H \smallsetminus \{1\}$,
and by $T_H$ the subset $\bigcup_{u \in X} H_u$ of $H^{(X)}$;
recall that $T_X$ is the subset of all transpositions in $\Sy (X)$.
Consider subsets $S_H \subset T_H$ and $S_X \subset T_X$,
and define $S$ to be the disjoint union $S_H \sqcup S_X$, inside $H \wr_X \Sy (X)$.
It is again elementary to check that
\begin{equation}
\label{GCwr}
\tag{GCwr}
\aligned
& \text{if $\Gamma(S_X)$ is connected and if
$S_H = \{ \varphi^{a_1}_{u_1}, \hdots, \varphi^{a_r}_{u_r} \}$}
\\
& \text{for some generating subset $\{a_1, \hdots, a_r\} \subset H$}
\\
& \text{and some sequence $u_1, \hdots, u_r$ of points of $X$.}
\\
&\text{then the group $H \wr_X \Sy (X)$ is generated by $S$.}
\endaligned
\end{equation}
\par

When $X$ is infinite, we consider subsets of $H \wr_X \Sy (X)$
of the form $S = S_H \sqcup S_X$ that satisfy the following condition:
\begin{equation}
\tag{PCwr}
\label{PCwr}
\aligned
& \text{the transposition graph $\Gamma(S_X)$ is connected and,}
\\
& \text{for all $L \ge 0$ and partition $\lambda = (\lambda_1, \hdots, \lambda_k) \vdash L$,}
\\
& \text{$\Gamma(S_X)$ contains a forest of $k$ trees $T_1, \hdots, T_k$,} 
\\
& \text{with $T_i$ having $\lambda_i$ vertices, including one of them, say $x^{(i)}$,}
\\
& \text{such that $\varphi^a_{x^{(i)}} \in S_H$ for all $a \in H \smallsetminus \{1\}$}.
\endaligned
\end{equation}
(The conditions ``for all $a \in H \smallsetminus \{1\}$''
could be replaced by ``for all $a$ in a set of representatives of the
conjugacy classes in $H$ distinct from $\{1\}$''.)

\begin{prop}
\label{wreath}
Let $H$ be a finite group;
denote by $M$ the number of conjugacy classes in $H$.
Consider an infinite set $X$, the wreath product $W = H \wr_X \Sy (X)$, 
and a generating subset $S$ that satisfies Condition \emph{(\ref{PCwr})}.
Then
$$
C_{W,S}(q) \, = \, \prod_{k=1}^\infty \frac{1}{(1 - q^k)^M}  \hskip.1cm .
$$
\end{prop}

Set $\prod_{k=1}^\infty \frac{1}{(1 - q^k)^M} = \sum_{n=0}^\infty p(n)_{(M)} q^n$.
For low values of the integer $M$, the sequences $\big( p(n)_{(M)} \big)_{n=0, 1, 2, \hdots}$
are well documented. 
For example, with A000041 and other similar numbers referring to those of \cite{OEIS},
we have:
$$
\aligned
&1, 1, 2, 3, 5, 7, 11, 15, 22, 30, 42, 
\hdots 
\hskip.2cm \text{for} \hskip.2cm (p(n)_{(1)})_{n \ge 0}, \hskip.2cm \text{see A000041;}
\\
&1, 2, 5, 10, 20, 36, 65, 110, 185, 300, 481, 
\hdots
\hskip.2cm \text{for} \hskip.2cm (p(n)_{(2)})_{n \ge 0}, \hskip.2cm \text{see A000712;}
\\
&1, 3, 9, 22, 51, 108,  221, 429, 810, 1479,  2640, \hdots
\hskip.2cm \text{for} \hskip.2cm (p(n)_{(3)})_{n \ge 0}, \hskip.2cm \text{see A000716;}
\\
&1,12, 90, 520, 2535, 10908, 42614, 153960, 
\hdots
\hskip.2cm \text{for} \hskip.2cm (p(n)_{(12)})_{n \ge 0}, \hskip.2cm \text{see A005758;}
\\
&\text{for} \hskip.2cm (p(n)_{(M)})_{n \ge 0}
\hskip.2cm  \text{when} \hskip.2cm  4 \le M \le 23, \hskip.2cm  M \ne 12, \hskip.2cm  \text{see A023003 to A023021.} 
\endaligned
$$
See also Section \ref{sectionRama} and Appendix \ref{AppendixC}
for some congruence relations satisfied by the coefficients 
$p(n)_{(M)}$.

\begin{proof}[Proof of Proposition \ref{wreath}]
In this proof, we write $G$ for $\Sy (X)$
and $W$ for $H \wr_X \Sy (X) = H^{(X)} \rtimes G$,

\vskip.2cm

\emph{Preliminary Remark.}
There are several ways to associate a conjugacy class in a symmetric group to a partition.
For example, when $X = \N$, 
in  Observation \ref{partconjclasses} above and many other places of this article,
the conjugay class associated to 
a  partition such as $(3,3,1) \vdash 7$ is that of
$$
(1,2,3,4)(5,6,7,8)(9,10) \in \Sy (\N) .
$$
In other places, in particular at some point of the present proof,
some fixed points of permutations are counted as parts of size $1$,
so that the conjugacy class associated to the same partition is that of\footnote{At this point,
it could be more consistent to include some fixed points in
cycle decompositions of permutations,
and thus to write $(1,2,3)(4,5,6)(7)(8)(9)\cdots \in \Sy (\N)$. }
$$
(1,2,3)(4,5,6) \in \Sy (\N) .
$$
This is the reason for which we use below one symbol, $\lambda$,
for a partition indexed by $1 \in H_*$ and a different symbol, $\mu$,
for a partition indexed by $\eta \ne 1$ in $H_*$.

\vskip.2cm

\emph{First step: reminder on the conjugacy classes of $W$.}
The set of conjugacy classes of $W$ is in bijection with the set of $H_*$-decorated partitions,
as we now describe, much as in \cite{Macd--95}.
Here, $H_*$ denotes the set of conjugacy classes of $H$;
we write $1 \in H_*$ rather than $\{1\} \in H_*$ for the class $\{1\} \subset H$.
\par

Let $w = (\varphi, f) \in H^{(X)} \rtimes_X \Sy (X)$.
We proceed to associate a $H_*$-indexed family of partitions
\begin{equation}
\tag{$\dagger$}
\label{eqno*}
\Big( \lambda^{(1)}, \big( \mu^{(\eta)} \big)_{\eta \in H_* \smallsetminus 1} \Big)
\end{equation}
to $w$.
\par

Let $X^{(w)}$ be the finite subset of $X$ that is the union of the supports of $\varphi$ and $f$.
Denote by $c_1, \hdots, c_k$ the disjoint cycles of which $f$ is the product.
Here, we include a cycle of length $1$ for each point $x \in X$ such that
$x \in \su (\varphi)$ and $x \notin \su (f)$, so that we have a disjoint union
$X^{(w)} = \bigsqcup_{1 \le i \le k} \su (c_i)$.
For $i \in \{1, \hdots, k\}$, 
there are points $x^i_j$ in $X^{(w)}$, with $1 \le j \le \nu_i := \text{length}(c_i)$,  such that
$c_i = (x^{(i)}_1, x^{(i)}_2, \hdots, x^{(i)}_{\nu_i})$.
Define $\eta^w_*(c_i) \in H_*$
to be the conjugacy class of the product
$\varphi(x^{(i)}_{\nu_i}) \varphi(x^{(i)}_{\nu_i - 1}) \cdots \varphi( x^{(i)}_1 ) \in H$.
Observe that the product itself is not well-defined by $c_i$,
since the $x^{(i)}_j$ are well-defined up to cyclic permutation only,
but that its conjugacy class is well-defined.
Observe also that, if $\nu_i = 1$, then  $\eta^w_*(c_i) \ne 1$.
\par

For $\eta \in H_*$ and $\ell \ge 1$, let $m^{w,\eta}_\ell$ denote the number
of cycles $c$ in $\{c_1, \hdots, c_k\}$ that are of length $\ell$
and are such that $\eta^w_*(c) = \eta$.
Let $\mu^{w,\eta}$ be the partition with $m^{w,\eta}_\ell$ parts equal to $\ell$,
for all $\ell \ge 1$;
let $n^{w,\eta}$ be the sum of the parts of this partition, so that
$\mu^{w,\eta} \vdash n^{w,\eta}$.
We have
$\sum_{\eta \in H_*} n^{w,\eta} = \sum_{\eta \in H_*, \ell \ge 1} \ell m^{w,\eta}_\ell
= \vert X^{(w)} \vert$.
\par

We define the \textbf{pretype} of $w$ as the family 
$\left( \mu^{w,\eta} \right)_{\eta \in H_*}$.
By a routine argument, it can now be checked that
\begin{enumerate}[noitemsep,label=(\roman*)]
\item\label{iDEwreath}
for all $w  = (\varphi, f) \in W$ and $g \in \Sy (X)$, 
\\  the pretypes of $w$ and $(1,g)w(1,g^{-1})$ coincide;
\item\label{iiDEwreath}
for all $w  = (\varphi, f) \in W$ and $\psi \in H^{(X)}$, 
\\
the pretypes of $w$ and $(\psi,1)w(\psi^{-1},1)$ coincide;
\end{enumerate}
hence conjugate elements in $W$ have the same pretype.
Moreover:
\begin{enumerate}[noitemsep,label=(\roman*)]
\addtocounter{enumi}{2}
\item\label{iiiDEwreath}
two elements in $W$ that have the same pretype are conjugate.
\end{enumerate}
For details, we refer to \cite[Appendix I.B, No.~3]{Macd--95}.
\par

For $w = (\varphi, f) \in W$, observe that the partition $\mu^{w,1}$
does not have parts of size $1$.
With the same notation as above, denote by $\lambda^{w,1}$
the partiton with $m^{w, 1}_\ell$ parts equal to $\ell - 1$.
We define the \textbf{type} of $w$ as the family 
$\big( \lambda^{w,1}, \big( \mu^{w,\eta} \big)_{\eta \in H_* \smallsetminus 1} \big)$.
Then \ref{iDEwreath} to \ref{iiiDEwreath} hold with ``type'' instead of ``pretype''.
Moreover:
\begin{enumerate}[noitemsep,label=(\roman*)]
\addtocounter{enumi}{3}
\item\label{ivDEwreath}
every $H_*$-indexed family of partitions, i.e., 
$\big( \lambda^{(1)}, \big( \mu^{(\eta)} \big)_{\eta \in H_* \smallsetminus 1} \big)$
as in (\ref{eqno*}), 
is the type of one conjugacy class in $W$.
\end{enumerate}

\vskip.2cm

\emph{Second step: proof of the formula for $C_{W,S} (q)$.}
Consider a $H_*$-index family of partitions
$\big( \lambda^{(1)}, \big( \mu^{(\eta)} \big)_{\eta \in H_* \smallsetminus 1} \big)$
as in (\ref{eqno*})
and the corresponding conjugacy class in $W$.
Denote by $n^{(1)}, n^{(\eta)}$ the sum of the parts 
and by $k^{(1)}, k^{(\eta)}$ the number of the parts
of $\lambda^{(1)}, \mu^{(\eta)}$, respectively.
Choose a representative $w = (\varphi, f)$ of this class, with $f$ of the form
$f = \prod_{i=1}^k c_i = \prod_{i=1}^k (x^{(i)}_1, x^{(i)}_2, \hdots, x^{(i)}_{\mu_i} )$
and
$$
\aligned
&\varphi( x^{(i)}_j) \, = \, 1 \in H
\hskip.2cm \text{for all} \hskip.2cm
j \in \{1, \hdots, \mu_i \}
\hskip1.6cm \text{when} \hskip.2cm
\eta^w_*(c_i) = 1 
\\
&\varphi( x^{(i)}_j) \, = \, 
\left\{
\aligned
1 \hskip.5cm &\text{for all} \hskip.2cm j \in \{1, \hdots, \mu_i - 1\}
\\
h \ne 1 \hskip.2cm &\text{for} \hskip.2cm j = \mu_i
\endaligned
\right.
\hskip.5cm \text{when} \hskip.2cm \eta^w_*(c_i) \ne 1 .
\endaligned
$$
Recall that $\eta^w_*(c_i) \ne 1$ when $\mu_i = 1$,
and observe that 
$$
\aligned
k \, &= \,  k^{(1)} + \sum_{\eta \in H_*, \eta \ne 1} k^{(\eta)}
\\
\vert X^{(w)} \vert \, &=  \,
n^{(1)} + k^{(1)} + \sum_{\eta \in H_*, \eta \ne 1} n^{(\eta)} .
\endaligned
$$
\par

The contribution of $(\varphi_{\vert \su(c_i)}, c_i)$ to $\kappa_{W,S}(q)$ is
$\mu_i - 1$ if $\eta^w_*(c_i) = 1$, and $\mu_i$ if $\eta^w_*(c_i) \ne 1$.
Hence, the contribution of the type
$\big( \lambda^{(1)}, \big( \mu^{(\eta)} \big)_{\eta \in H_* \smallsetminus 1} \big)$
to $C_{W,S}(q)$ is $q^{n^{(1)}} \prod_{\eta \in H_*, \eta \ne 1} q^{ n^{(\eta)}}$.
It follows that 
$$
\aligned
C_{W,S}(q) \, &= \,
\Big( \sum_{n_1 = 0}^\infty p(n_1) q^{n_1} \Big)
\prod_{\eta \in H_*, \eta \ne 1}
\Big( \sum_{n_\eta = 0}^\infty p(n_\eta) q^{n_\eta} \Big)
\\
\, &= \,
\prod_{k=1}^\infty \frac{  1  }{ (1-q^k)^{ \vert H_* \vert }  } .
\endaligned
$$
This ends the proof of Proposition \ref{wreath}.
\end{proof}

\section{\textbf{A sample of growth polynomials and conjugacy growth polynomials 
for finite symmetric groups}}
\label{sectionLCsym}

The purpose of the present section is to compute for $\Sy (n)$
growth polynomials $L_{\Sy (n), S}(q)$
and conjugacy growth polynomials $C_{\Sy (n), S}(q)$,
with respect to a sample of generating sets $S$.
Our computations rely partly on 
Lemmas \ref{lemmafcl} of Section \ref{explicit...series}
and \ref{longS0equalfinite} of Section \ref{sectionproofs}.
\par

Before this, we review part of what is known 
in the broader and classical setting of finite Coxeter groups.
Though we will not recall precise statements,
this is strongly related to the topology of connected compact Lie groups
and their homogenous spaces.
\par

Let $(W,S)$ be a finite Coxeter system; set $l = \vert S \vert$.
Denote the corresponding Coxeter exponents
by $m_1, \hdots, m_l$; they are positive integers.
The growth polynomial is known to be 
\begin{equation}
\tag{$L_{W,S}$}
\label{LWS}
L_{W,S} (q) \, = \  \prod_{k=1}^l (1 + q + \cdots + q^{m_k} ) .
\end{equation}
This has received much attention;
see for example \cite{Solo--66} and \cite[exercises of $\S$ IV.1 and VI.4]{Bour--68}.
As Solomon observes, the computation of $L_{W,S}$ for the particular case
of the symmetric groups goes back to Rodrigues, in the first half of XIXth century
(with a different formulation).
Set $T = \bigcup_{w \in W} wSw^{-1}$.
The word length $\ell_{W,T}$ 
is sometimes called the \emph{reflection length} \cite{Cart--72}
and the corresponding growth polynomial is known to be 
\begin{equation}
\tag{$L_{W,T}$}
\label{LWT}
L_{W, T} (q) \, = \,  \prod_{k=1}^l (1 + m_k q) .
\end{equation}
For a group $W \subset \operatorname{GL}(V)$ generated by reflections,
define $\rho : W \longrightarrow \N$ by 
$$
\rho(w) = \dim (V) - \dim (\{v \in V \mid w(v) = v\})
$$
and set $R_W(q) = \sum_{w \in W} q^{\rho(w)}$.
Then $R_W(q) = \prod_{k=1}^l (1 + m_k q)$;
this is a special case of \cite[Number 5.3]{ShTo--54},
verified there by inspection,
and shown again more conceptually in \cite{Solo--63}.
For a finite \emph{Weyl group}, it is easy to show that $\rho(w) = \ell_{W,S}(w)$,
see e.g.\ \cite[Lemma 2]{Cart--72},
so that $L_{W,T} = R_W$, and (\ref{LWT}) holds;
this carries over to every finite \emph{Coxeter group},
see e.g.\ \cite{Lehr--87}.
Other avatars of these polynomials are discussed in \cite{BaGo--94}.
\par

We do not know whether 
the companion polynomials $C_{W,S}, C_{W,T}$ have already been given any attention.
\par

In the next proposition, 
we particularize $L_{W,S}(q)$ and $L_{W,T}(q)$ to $W = \Sy (n)$,
and we provide expressions for the corresponding conjugacy growth polynomials.
In the special case of finite symmetric groups,
there is an ad hoc proof for (\ref{LWT}) in Remark \ref{badwritingforC}
and one for (\ref{LWS}) in \cite{Harp--91}.

\begin{prop}
\label{rodrigues}
Consider an integer $n \ge 1$, the symmetric group $\Sy (n)$
and its generating sets
$$
\aligned
S^{\text{Cox}}_n \, &= \,  \{ (1,2), (2,3) , \cdots , (n-1,n) \} ,
\\
T_n \, &= \,  \{ (i,j) \mid 1 \le i,j \le n,\hskip.1cm  i < j \} ,
\endaligned
$$
as in Proposition \ref{firstsample}.
The corresponding growth polynomial and conjugacy growth polynomial are
$$
\aligned
L_{\Sy (n), S^{\text{Cox}}_n}(q) \, &= \, \prod_{k=1}^{n-1} (1 + q + \cdots + q^k) ,
\\
L_{\Sy (n), T_n}(q) \, &= \, \prod_{k=1}^{n-1} (1 + k q) ,
\\
C_{\Sy (n), S^{\text{Cox}}_n}(q) \, &= \, C_{\Sy (n), T_n} (q) \, = \, 
 \sum_{k=0}^{n-1} p_{n-k}(n) q^k  \hskip.1cm ,
\endaligned
$$
where $p_{n-k}(n)$ is as in Appendix \ref{recallpk(n)}.
\end{prop}

\begin{proof}
The equalities involving the two products are particular cases
of (\ref{LWS}) and (\ref{LWT}), since the Coxeter exponents of
$(\Sy (n), S^{\text{Cox}}_n)$ are $1, 2, \hdots, n-1$.
The equality for $C_{\Sy (n), S^{\text{Cox}}_n}(q)$
is that of Proposition \ref{firstsample}\ref{iiDEfirstsample},
and $C_{\Sy (n), T_n} (q)$ is the same polynomial,
see Remark \ref{CompareSandT}.
\end{proof}

The polynomials $C_{\Sy (n), T_n}(q)$ for small $n$ 's are given by
$$
\aligned
C_{\Sy (2), T_2}(q) \, &= \,  1+q ,
\\
C_{\Sy (3), T_3}(q) \, &= \,   1 + q + q^2 ,
\\
C_{\Sy (4), T_4}(q) \, &= \,   1 + q +2 q^2 + q^3 , 
\\
C_{\Sy (5), T_5}(q) \, &= \,   1 + q + 2 q^2 + 2 q^3 + q^4 ,
\\
C_{\Sy (6), T_6}(q) \, &= \,   1 + q + 2 q^2 + 3 q^3 + 3q^4 + q^5 \hskip.1cm .
\endaligned
$$
(Compare with the polynomials written after Proposition \ref{growthpolSym(n)S0}.)

\begin{rem}
\label{badwritingforC}
(i)
The second polynomial of Proposition \ref{rodrigues}
can also be written
$$
L_{\Sy (n), T_n}(q) \, = \,
1 + \sum_{m=2}^n \binom{n}{m} \sum_{k=1}^{\lfloor m/2 \rfloor} d_k(m) q^{m-k}  ,
$$
where $d_k(m)$ is as in Appendix \ref{recalldk(n)}.

\vskip.2cm

(ii)
It is easy to check directly from (i) that we have also
$$
L_{\Sy (n), T_n}(q) \, = \,
\prod_{k=1}^{n-1} (1 + kq) ,
$$
as in Proposition \ref{rodrigues}.
\end{rem}

\begin{proof}
(i)
For $m \in \{0, 1, \hdots, n\}$, there are $\binom{n}{m}$ subsets of size $m$ in
$\{1, 2, \hdots, n\}$.
For each such subset, say $A$, and each $k \in \{0, 1, \hdots, n\}$,
there are $d_k(m)$ permutations in $\Sy (n)$ with support $A$
which are products of $k$ disjoint cycles,
and these elements have $T_n$-word length $m-k$, by Lemma \ref{lemmafcl}.
The growth polynomial of the situation is therefore
$$
\sum_{m=0}^n \binom{n}{m} \sum_{k=0}^n d_k(m) q^{m-k} \hskip.1cm .
$$
To end this computation, we observe that
the contribution of $m=0$ is $1$,  that of $m=1$ is $0$,
and $d_k(m) = 0$ for $2k > n$.

\vskip.2cm

(ii)
We proceed by induction on $n$.
There is nothing to check for $n=1$;
we assume now that $n \ge 2$, 
and that the statement holds for $n-1$.
\par

Consider an element $g \in \Sy (n)$ which is not in $\Sy (n-1)$.
There is a unique pair consisting of
$i \in \{1, \hdots, n-1\}$ and $h \in \Sy (n-1)$ 
such that $g = (i,n) h$.
This implies that
$$
C_{\Sy (n), T_n}(q) \, = \, 
C_{\Sy (n-1), T_{n-1}}(q) + (n-1)qC_{\Sy (n-1), T_{n-1}}(q) .
$$ 
Hence
$$
C_{\Sy (n), T_n}(q) \, = \, C_{\Sy (n-1), T_{n-1}}(q) \hskip.1cm \big(1 + (n-1)q \big)
\, = \, \prod_{i=1}^{n-1} (1 + kq)
$$
by the induction hypothesis.
\end{proof}

The final proposition of this section shows polynomials $L$ and $C$
for finite symmetric groups and a third generating set $S^0_n$,
essentially distinct from the generating sets $S^{\text{Cox}}_n$ and $T_n$
of Proposition \ref{rodrigues} for $n \ge 4$.
It is convenient to see $\Sy (n)$
as the symmetric group of $\{0, 1, \hdots, n-1\}$;
the generating set $S^0_n$ is that already considered in
Lemmas \ref{longS0smaller} and~\ref{longS0equalfinite}.

\begin{prop}
\label{growthpolSym(n)S0}
Consider an integer $n \ge 1$, the symmetric group $\Sy (n)$
and the generating set
$S^0_n = \{ (0, i) \mid 1 \le i \le n-1 \}$.
The corresponding growth polynomial and conjugacy growth polynomial are
$$
\aligned
L_{\Sy (n), S^0_n}(q) \, &= \, 
1 + \sum_{m=2}^{n-1} \binom{n-1}{m}
\sum_{k=1}^{\lfloor m/2 \rfloor} d_k(m) q^{m+k}
\\
&+ \,  
\sum_{m=2}^n \binom{n-1}{m-1}
\sum_{k=1}^{\lfloor m/2 \rfloor} d_k(m) q^{m+k-2}
\hskip.1cm ,
\\
C_{\Sy (n), S^0_n}(q) \, &= \, 1 \, + \, 
\sum_{k=1}^{\lfloor n/2 \rfloor} q^{2k-2} \sum_{j=k}^n p_k(j)q^{j}
\hskip.2cm \text{(as in Proposition \ref{propS0}).}
\endaligned
$$
\end{prop}

For example: 
$$
\aligned
L_{\Sy (4), S^0_4} (q) \, &= \, 1 + 3q + 6q^2 + 9q^3 + 5q^4 ,
\\
L_{\Sy (5), S^0_5} (q) \, &= \, 1 + 4q + 12 q^2 + 30q^3 + 44q^4 + 26 q^5 + 3q^6 ,
\\
L_{\Sy (6), S^0_6} (q) \, &= \, 
1 + 5q + 20q^2 + 70q^3 + 170 q^4 + 250 q^5 + 169q^4 + 35q^7 ,
\endaligned
$$
and
$$
\aligned
C_{\Sy (4), S^0_4}(q) \, &= \,   1 + q + q^2 + q^3 + q^4 , 
\\
C_{\Sy (5), S^0_5}(q) \, &= \,   1 + q + q^2 + q^3 + 2q^4 + q^5 ,
\\
C_{\Sy (6), S^0_6}(q) \, &= \,   1 + q + q^2 + q^3 + 2q^4 + 2q^5 + 2q^6 + q^7 \hskip.1cm .
\endaligned
$$
(Compare with the polynomials written after Proposition \ref{rodrigues}.)

\begin{proof}
Let us deal with the polynomial $L$.
Consider first elements $g \in \Sy (n)$ with $0 \notin \su (g)$.
For each $m \in \{0, 1, \hdots, n-1\}$,
there are $\binom{n-1}{m}$ subsets of size $m$ in $\{1, 2, \hdots, n-1\}$.
For each such subset, say $A$, and each $k \in \{0, 1, 2, \hdots, m\}$,
there are $d_k(m)$ elements with support $A$
which are products of $k$ cycles, and these elements
have $S^0_n$-word length $m+k$, by Lemma \ref{longS0equalfinite}.
The contribution to the growth polynomial
of elements with $0 \notin \su (g)$ is therefore
\begin{equation}
\tag{$0 \notin \su $}
\label{eq1DEgrowthpolSym(n)S0}
\sum_{m=0}^{n-1} \binom{n-1}{m} \sum_{k=0}^m d_k(m) q^{m+k}  \hskip.1cm .
\end{equation}
The contribution of $m=0$ is $1$ and that of $m=1$ is zero;
for $m \ge 2$, the contributions of terms with $k=0$ or $k>m/2$ is also zero.
\par

Consider now elements $g \in \Sy (n)$ with $0 \in \su (g)$.
For each $m \in \{1, 2, \hdots, n\}$,
there are $\binom{n-1}{m-1}$ subsets of size $m$ in $\{0, 1, \hdots, n-1\}$
containing $0$.
For each such subset, say $B$, and each $k \in \{1, 2, \hdots, m\}$,
there are $d_k(m)$ elements with support $B$
which are products of $k$ cycles, and these elements
have $S^0_n$-word length $m+k-2$.
The contribution of these elements is therefore
\begin{equation}
\tag{$0 \in \su $}
\label{eq2DEgrowthpolSym(n)S0}
\sum_{m=1}^{n} \binom{n-1}{m-1} \sum_{k=1}^m d_k(m) q^{m+k-2}  \hskip.1cm .
\end{equation}
As above, the contributions of terms with $m=1$ or $k>m/2$ vanish.
\par
The formula for $L_{\Sy (n), S^0_n}(q)$ follows.
That for $C_{\Sy (n), S^0_n}(q)$ is
a repetition of part of Proposition \ref{propS0}.
\end{proof}

\section{\textbf{Alternating groups}}
\label{sectionalt}

For a non-empty set $X$, we denote by $\Alt (X)$ the \textbf{finitary alternating group} of $X$,
i.e.\ the subgroup of $\Sy (X)$ of permutations of even signature.
Set
$$
\aligned
T^A_X \, &= \, \big\{ (x,y,z) \in \Alt (X) \mid x,y, z \in X \hskip.2cm \text{are distinct} \big\} ,
\\
U^A_X \, &= \, \big\{ (x,y)(z,u) \in \Alt (X) \mid x,y,z,u \in X \hskip.2cm \text{are distinct} \big\} .
\endaligned
$$
Recall from the introduction that, when $X = \N$, we have defined
$$
S^A_\N \, = \, \{ (i, i+1, i+2) \in \Alt (\N) \mid i \in \N \} ,
$$
and we consider also
$$
R^A_\N \, = \, \big\{ (1, i, i+1) \in \Alt (\N) \mid i \ge 2 \big\} .
$$
When $X = \{1, \hdots n\}$ for some $n \ge 3$, we write
$$
\aligned
\Alt (n) \, &= \, \Alt ( \{1, 2, \hdots, n\} ) ,
\\
S^A_n \, &= \,  \{ (i, i+1, i+2) \in \Alt (n) \mid 1 \le i \le n-2 \} ,
\\
R^A_n \, &= \,  \{ (1, i, i+1) \in \Alt (n) \mid 2 \le i \le n-1 \} .
\endaligned
$$
When $X$ is either $\N$ or $\{1, \hdots, n\}$ for some $n \ge 1$,
we write $S^A_X$ to denote the relevant set, either $S^A_\N$ or $S^A_n$,
and similarly for $R^A_X$.
\par

The following lemma is well-known,
even if we did not find a convenient reference.

\begin{lem}
\label{sgenalt}
With the notation above:
\par
for all $n \ge 3$, the sets $S^A_n$ and $R^A_n$ both generate $\Alt (n)$;
\par
the sets $S^A_\N$ and $R^A_\N$ both generate $\Alt (\N)$;
\par
and the set $T^A_X$ generates $\Alt (X)$.
\end{lem}

\begin{proof}
Let $H_n$ denote the subgroup of $\Alt (n)$ generated by $S^A_n$;
we claim that $H_n = \Alt (n)$. 
The case of $n=3$ is obvious; we proceed by induction on $n$,
assuming that $n \ge 4$ and that the claim holds for $n-1$.
\par

The group $H_n$ acts transitively on $\{1, \hdots, n\}$,
because it contains the $3$-cycle $(n-2, n-1, n)$ as well as $H_{n-1} = \Alt (n-1)$.
Hence the order of $H_n$ is $n$ times the index of
the isotropy group $\{h \in H_n \mid h(n) = n\}$,
that is $\vert H_n \vert = n \frac{1}{2} (n-1)! = \frac{1}{2} n!$.
It follows that $H_n = \Alt (n)$.
\par

As a consequence, $R^A_n$ also generates $\Alt (n)$, 
since 
\hfill\par\noindent
$(1, i+1, i)(1, i+2, i+1)(1, i, i+1) = (i, i+1, i+2)$ 
for all $i \in \{2, \hdots, n-1\}$.
\par

The claims for $S^A_\N$, $R^A_\N$ and $T^A_X$ follow.
\end{proof}

Note that $T^A_X \cup U^A_X$ is the set of products of two distinct elements
of the generating set $S_X$ of $\Sy (X)$. 
It follows that
$$
\kappa_{\Alt (X), T^A_X \cup U^A_X}(g)
\, = \,
\frac{1}{2} \kappa_{\Sy (X), S_X}(g)
\hskip.5cm \text{for all} \hskip.2cm g \in \Alt (X) .
$$
Since, for $X$ infinite, two elements of $\Alt (X)$ are conjugate in $\Alt (X)$
if and only if they are conjugate in $\Sy (X)$, 
we obtain the following straightforward consequence of Proposition \ref{sharpening}:

\begin{prop}
Let $X$ be an infinite set, and $T^A_X, U^A_X$ as above. Then
$$
\aligned
C_{\Alt (X), T^A_X \cup U^A_X}(q) \, &= \, \sum_{m=0}^\infty p(2m) q^m 
\\
\, &= \, 
1 + 2q + 5q^2 + 11q^3 + 22q^4 + 42 q^5 + 77 q^6 
\\
\, &+ \, 135 q^7 + 231 q^8 + 385 q^9 + 627 q^{10}
+ 1002 q^{11} + 1575 q^{12} + \cdots ,
\endaligned
$$
where the numerical coefficients are those of the series \cite[A058696]{OEIS}.
\end{prop}

Let $X$ be a set containing at least $5$ elements.
It is easy to check that $U^A_X$ generates $\Alt (X)$,
and it can be shown that 
$$
\kappa_{\Alt (X), U^A_X}(g) \, = \,
\left\{ \aligned
\kappa_{\Alt (X), T^A_X \cup U^A_X}(g) \hskip.5cm &\text{if} \hskip.2cm
   g = \operatorname{id} \hskip.2cm \text{or} \hskip.2cm \vert \su (g) \vert > 3 ,
\\
2 \hskip2cm &\text{if} \hskip.2cm
   g \hskip.2cm \text{is a $3$-cycle}.
\endaligned \right.
$$
It follows that

\begin{prop}
Let $X$ be an infinite set, and $U^A_X$ as above. Then
$$
C_{\Alt (X), U^A_X}(q) \, = \, q^2 - q + \sum_{m=0}^\infty p(2m) q^m .
$$
\end{prop}

\emph{Remark.}
For the generating set
$$
V^A_\N \, := \, \{(i, i+1, i+2) \in \Alt (\N) \mid i \ge 0 \} \cup \{ (i,i+1)(i+2,i+3) \in \Alt (\N) \mid i \ge 0 \} ,
$$
it can be shown that
$$
C_{\Alt (\N), V^A_\N}(q) \, = \, C_{\Alt (\N), T^A_\N \cup U^A_\N}(q) .
$$

\vskip.2cm

Our next target is to identify $C_{\Alt (X), T^A_X}(q)$.

\begin{lem}
\label{susinsupg}
Let $g \in \Alt (X)$ and $g = t_1 \cdots t_L$
a writing of $g$ as a word of minimal length $L = \ell_{\Alt (X), T^A_X}(g)$
in the generators of $T^A_X$.
\par
Then $t_j \ne t_i^{\pm 1}$, 
equivalently $\su (t_i) \ne \su (t_j)$, for all $i,j \in \{1, \hdots, L\}$ with $i \ne j$.
\end{lem}

\begin{proof}
Let $g = u_1 \cdots u_M$ be a writing of $g$ as a word in the generators of $T^A_X$.
\par
Suppose first that there exist $j,k \in \{1, \hdots, M\}$ with $j < k$ 
such that $u_k = u_j^{-1}$.
If $k=j+1$, then deleting $u_ju_k$ 
produces a new $T^A_X$-word of length $M-2$ representing $g$;
if $k \ge j+2$, then $g$ can be written as 
$$
u_1 \cdots u_{j-1} 
\left(u_j u_{j+1} u_j^{-1}\right) \cdots \left( u_j u_{k-1} u_j^{-1}\right)
u_{k+1} \cdots u_m ,
$$
i.e.\ $g$ can again be written as a $T^A_X$-word of length $M-2$ representing $g$.
\par

Suppose now that there exist $j,k \in \{1, \hdots, M\}$ with $j<k$ such that $u_k = u_j$.
If $k=j+1$, then replacing $u_ju_k$ by $u_j^{-1}$ 
produces a new $T^A_X$-word of length $M-1$ representing $g$;
if $k \ge j+2$, then $g$ can be written as
$$
u_1 \cdots u_{j-1} u_j u_{j+1} \cdots u_{k-1} u_j^{-1} u_j^{-1} u_{k+1} \cdots u_M
$$
and the previous procedure provides 
a $T^A_X$-word representing $g$ of length $M-1$.
\par
The lemma follows.
\end{proof}

For $g \in \Alt (X)$ a product of disjoint cycles,
we denote by $k'_g$ the number of cycles of odd lengths $\ge 3$
and by $2k''_g$ the number of cycles of even lengths $\ge 2$.
Note that $k_g = k'_g + 2k''_g$ for $k_g$ as in Lemma \ref{lemmalimit}.

\begin{lem}
\label{conjlengthalt}
Let $X$ be a set and $S$ a generating set of $\Alt (X)$.
Let $g  \in \Alt (X)$ be a product of disjoint cycles, with $k'_g, k''_g$ as above.
Suppose either that $S = T^A_X$ or that $X$ is one of $\N$, $\{1, \hdots, n\}$
for some $n \ge 1$, and that $S^A_X \subset S \subset T^A_X$.
We have
$$
\ell_{\Alt (X), T^A_X} (g) \, = \, 
\kappa_{\Alt (X), S}(g) \, = \, 
\frac{1}{2} \big( \vert \su (g) \vert - k'_g \big) .
$$
\end{lem}

In the proof below, we write $\ell$ for $\ell_{\Alt (X), T^A_X}$
and $\kappa$ for $\kappa_{\Alt (X), S}$.

\begin{proof}[Proof of the upper bounds
$\kappa (g), \ell (g)  \le \frac{1}{2} \big( \vert \su (g) \vert - k'_g \big)$]
We show the bound for $\ell (g)$,
and leave it to the reader to check that a minor modification
of the same argument shows the bound for $\kappa (g)$.
Whenever convenient,
we write $k', k''$ rather than $k'_g, k''_g$.
\par
Consider a cycle of odd length, say
$$
c_\alpha \, = \,  (x_1, x_2, \hdots, x_{2p+1}) 
$$
for $x_1, \hdots, x_{2p+1} \in X$.
We have
$$
c_\alpha \ = \, 
(x_1, x_2, x_3) (x_3, x_4, x_5) (x_5, x_6, x_7) \cdots  (x_{2p-1}, x_{2p}, x_{2p+1}) 
$$
and therefore $\ell (c_\alpha) \le p = \frac{1}{2} (\vert \su (c_\alpha) \vert - 1)$.
\par

Consider a pair of disjoint cycles of even lengths, say
$$
c_\beta c_\gamma \, = \,
(x_1, x_2, \hdots, x_{2r} )(y_1, y_2, \hdots, y_{2s} ) 
$$
for $x_1, \hdots, x_{2r}, y_1 \hdots, y_{2s} \in X$
(where we consider an appropriate conjugate of $g$
and $2r+2s$ \emph{consecutive} integers 
$y_1,y_2,\dots,y_{2s},x_1,x_2,\dots,x_{2r}$ 
for the case of $\kappa (g)$).
We have
$$
\aligned 
c_\beta c_\gamma \, = \,
&(y_1, y_2, y_3) (y_3, y_4, y_5)  \cdots (y_{2s-3}, y_{2s-2}, y_{2s-1}) (y_{2s-1}, y_{2s}, x_1)
\\
&(x_1, x_2, x_3) (x_3, x_4, x_5) \cdots (x_{2r-3}, x_{2r-2}, x_{2r-1}) (x_{2r-1}, x_{2r}, y_{2s}) 
\endaligned
$$
and therefore    
$\ell (c_\beta c_\gamma) \le r+s = 
\frac{1}{2}\left(  \vert \su(c_\beta) \vert + \vert  \su(c_\gamma) \vert \right)$.
\par

For $g = c_1 c_2 \cdots c_{k'} c_{k'+1} c_{k'+2} \cdots c_{k'+2k''}$,
where $c_1, \hdots, c_{k'+2k''}$ are disjoint cycles,
$c_\nu$ of odd length for $1 \le \nu \le k'$ 
and of even length for $k'+1 \le \nu \le k'+2k''$, 
it follows that
$$
\aligned
\ell (g) \, \le \, \sum_{\nu=1}^{k'+2k''} \ell (c_\nu) 
& \, \le \,
\frac{1}{2} \Big( \sum_{\alpha=1}^{k'} (\vert \su (c_\alpha) \vert - 1)
+ \sum_{\beta=k'+1}^{2k''} \vert \su (c_\beta) \vert  \Big) 
\\
& \, = \, 
\frac{1}{2} \big( \vert \su (g) \vert - k' \big) \hskip.1cm ,
\endaligned
$$
as was to be shown.
\end{proof}

\begin{proof}[Proof of the lower bounds
$\ell (g), \kappa (g) \ge \frac{1}{2} \big( \vert \su (g) \vert - k'_g \big)$]
For $g \ne \operatorname{id}$ in $\Alt (X)$ such that $\vert \su (g) \vert \le 3$,
we have obviously 
$1 = \ell (g) = \kappa (g) \ge \frac{1}{2} \big( \vert \su (g) \vert - k'_g \big) = \frac{1}{2}(3-1)$.
We consider from now on an element $g$ in $\Alt (X)$ with 
$\vert \su (g) \vert > 3$, and therefore with $\ell (g) > 1$ and $\kappa (g) > 1$.
As above, we continue and deal with $\ell (g)$ only.
\par

Suppose by contradiction that there exists $g \in \Alt (X)$ with
$\vert \su (g) \vert > 3$ and
\begin{equation}
\tag{$\flat$}
\label{flat}
\ell(g) \, < \,  \frac{1}{2} \big( \vert \su (g) \vert - k'_g \big) \hskip.1cm ;
\end{equation}
suppose moreover that $\ell (g)$ is minimal for the elements 
for which (\ref{flat}) holds.
We can write 
\begin{equation}
\tag{$\flat\flat$}
\label{flatflat}
g = t_1 \cdots t_L
\end{equation} 
for some $t_1, \hdots, t_L \in T^A_X$ with 
$
1 \, < \, L \, = \, \ell(g) \, < \, \frac{1}{2} \big( \vert \su (g) \vert - k'_g \big) \hskip.1cm .
$
By Lemma \ref{susinsupg}, we know that
the supports $\su (t_i)$ are pairwise distinct.
\par

For each $i \in \{1, \hdots, L\}$, let $x_i, y_i, z_i \in X$ be such that
$t_i = (x_i, y_i, z_i)$.
Set $Y_i = \su (t_i) = \{x_i, y_i, z_i\}$ and 
$Z_i = \bigcup_{1 \le j \le L, \hskip.1cm j \ne i} Y_j$.

\vskip.2cm

\emph{Claim:}
We have
\begin{equation}
\tag{$\sharp$}
\label{sharp}
\vert Y_i \cap Z_i \vert \, \ge \, 2 \hskip.2cm \text{for all} \hskip.2cm
i \in \{1, \hdots, L\} .
\end{equation}
\par

Upon conjugating $g$ by $t_{i+1} \cdots t_L$,
we can assume that $i=L$  for the proof of the claim. 
\par

Let us first check that 
$\vert Y_L \cap Z_L \vert \ge 1$.
Indeed, otherwise, set 
\begin{equation}
\tag{$\ddagger$}
\label{ddagger}
h \, = \,  \prod_{i=1}^{L-1} t_i   \hskip.1cm .
\end{equation}
Observe that $\ell (h) \le L-1$.
We have $\vert \su (h) \vert = \vert \su (g) \vert -3$,
and also $k'_h = k'_g -1$, since the cycle of odd length $t_i$
has been deleted in the product defining $h$.
It follows that 
$\ell (h) < \frac{1}{2} \big( \vert \su (h) \vert - k'_h \big)$.
This contradicts the minimality hypothesis on $g$ made above; 
hence $\vert Y_L \cap Z_L \vert \ge 1$.
\par

Let us now show that
$\vert Y_L \cap Z_L \vert \ge 2$.
Indeed, otherwise,  $\vert Y_L \cap Z_L \vert = 1$.
Let again $h$ be defined by (\ref{ddagger});
observe again that $\ell (h) \le L-1$,
and that $\vert \su (h) \vert = \vert \su (g) \vert -2$;
it can be shown that $k'_h = k'_g$ (details below).
It follows that $\ell (h) < \frac{1}{2} \big( \vert \su (h) \vert - k'_h \big)$.
This contradicts again the minimality hypothesis above;
hence $\vert Y_L \cap Z_L \vert \ge 2$.
\par

Here are the announced details.
Let $x,y,z \in X$ be such that $Y_L \cap Z_L = \{x\}$ and $t_L = (x,y,z)$;
Then $x$ is contained in the support of a cycle $d$ of $h$ of length $\ell \ge 2$,
and also by Lemma \ref{twocycles}
in the support of a cycle $c = dt_L$ of $g = ht_L$ of length $\ell + 2$.
Hence $k'_h = k'_g$.
\par

This ends the proof of the Claim.

\vskip.2cm

Lemma \ref{susinsupg} and the claim just proven imply that,
for each $i \in \{1, \hdots, L\}$, 
there are $x_i, y_i, z_i \in X$ such that
\begin{enumerate}[noitemsep,label=(\arabic*)]
\item[]\label{1DEtruc}
$t_i = (x_i, y_i, z_i)$, 
\item[]\label{2DEtruc}
$y_i, z_i \in Z_i$.
\end{enumerate}
Consider the product of $2L$ transpositions, equal to $g$,
obtained from the product (\ref{flatflat})
by changing each $t_i$ to $(x_i, z_i) (x_i, y_i)$,
say
$$
g \, = \, s_1 s_2 \cdots s_{2L-1} s_{2L} .
$$
\par

Set $S = \{s_1, \hdots, s_{2L} \}$;
define $\widetilde \Gamma (S)$ to be the multigraph
with vertex set $V := \bigcup_{j=1}^{2L} \su (s_j)$, 
and one edge connecting $x, y \in V$ for every 
$j \in \{1, \hdots, 2L\}$ with $s_j = (x, y)$;
here, ``multigraph'' means that $\widetilde \Gamma (S)$
may have multiple edges.
On the one hand, 
the number of vertices of this graph is bounded below
by $\vert \su (g) \vert$;
on the other hand, what we have shown so far implies that 
the degree of each vertex of $\widetilde \Gamma (S)$ is at least $2$;
it follows that the number of edges of this graph,
which is at least twice its number of vertices,
is bounded below by $\vert \su (g) \vert$;
in other words, $L \ge \frac{1}{2} \vert \su (g) \vert$.
This is strongly in contradiction with (\ref{flat});
hence the inequality of  (\ref{flat}) is not true,
and this ends the proof of the lemma.
\end{proof}

\emph{Remark concerning the claim of the previous proof.}
Consider an element $g \in \Alt (X)$ which is a word $g=t_1 \cdots t_L$
in the letters of $T^A_X$ of minimal length $L = \ell (g)$,
now with $2 \le L \le \frac{1}{2} ( \su (g) \vert - k'_g)$.
The cardinality $\vert Y_1 \cap Z_1 \vert$ can be any of $0,1, 2, 3$,
as the following examples show:
$$
\aligned
g_0 \, = \,  &(1, 2, 3)(4,5,6)
\hskip.2cm \text{for which} \hskip.2cm
\\
&L=2 ,  \hskip.2cm
\vert \su (g_0) \vert - k'_{g_0} = 6-2 ,
\hskip.2cm \text{and} \hskip.2cm
Y_1 \cap Z_1 = \emptyset ,
\\
g_1 \, = \,  &(1, 4, 5) (1, 2, 3) = (1, 2, 3, 4, 5)
\hskip.2cm \text{for which} \hskip.2cm
\\
&L=2 ,  \hskip.2cm
\vert \su (g_1) \vert - k'_{g_1} = 5-1 ,
\hskip.2cm \text{and} \hskip.2cm
Y_2 \cap Z_2 = \{1\} ,
\\
g_2 \, = \,  &(5, 6, 7) (2, 3, 4) (1, 4, 7) = (1, 2, 3, 4, 5, 6, 7)
\hskip.2cm \text{for which} \hskip.2cm
\\
&L=3 ,  \hskip.2cm
\vert \su (g_2) \vert - k'_{g_2} = 7-1 ,
\hskip.2cm \text{and} \hskip.2cm
Y_3 \cap Z_3 = \{4, 7\} ,
\\
g_3\, = \,  &(1, 8, 9) (5,6,7) (2,3,4) (1, 4, 7) = (1, 2, 3, 4, 5, 6, 7, 8, 9)
\hskip.2cm \text{for which} \hskip.2cm
\\
&L=4 ,  \hskip.2cm
\vert \su (g_3) \vert - k'_{g_3} = 9-1 ,
\hskip.2cm \text{and} \hskip.2cm
Y_4 \cap Z_4 = \{1, 4, 7\} .
\endaligned
$$

\vskip.2cm

Proposition \ref{CforALt(N)}
is a minor generalization of Proposition \ref{AltIntro}.
Recall from Appendix \ref{recallpepo(n)} that $p_e(n)$ 
denotes the number of partitions of $n \in \N$ involving an even number of positive parts.

\begin{prop}
\label{CforALt(N)}
Let $X$ be an infinite set and $S$ a generating set of $\Alt (X)$.
Suppose either that $S = T^A_X$ or that $X = \N$
and that $S^A_\N \subset S \subset T^A_\N$.
Then
$$
\aligned
C_{\Alt (X), S} (q) \, &= \, 
\sum_{u=0}^\infty p(u) q^u \sum_{v=0}^\infty p_e(v) q^v
\\
&= \, \frac{1}{2} \prod_{j=1}^\infty \frac{1}{(1-q^j)^2} \, + \,
\frac{1}{2} \prod_{j=1}^\infty \frac{1}{1-q^{2j}}
\\
&= \, 1+q+3q^2+5q^3+11q^4+18q^5+34q^6
\\
& \hskip.5cm +55q^7+95q^8+150q^9+244q^{10} + \cdots \hskip.1cm .
\endaligned
$$
\end{prop}

\begin{proof}
We write $\kappa$ for $\kappa_{\Alt (X), S}$.
\par

Let $g \in \Alt (X)$ be written as a product of disjoint cycles,
say $k'$ of them of odd lengths and $2k''$ of them of even lengths.
Denote by $g_o$ the product of the cycles of odd lengths
and by $g_e$ the product of the cycles of even lengths,
so that $g = g_o g_e$. 
Let $\lambda^{(g)} = (\lambda^{(g)}_1, \hdots, \lambda^{(g)}_{k'}) \vdash u$ 
and $\nu^{(g)} = (\nu^{(g)}_1, \hdots, \nu^{(g)}_{2k''}) \vdash v$
be the partitions such that 
$g_o$ is the product of cycles of lengths
$2\lambda^{(g)}_1+1, \hdots, 2\lambda^{(g)}_{k'}+1$, 
and $g_e$ the product of  cycles of lengths
$2\nu^{(g)}_1, \hdots, 2\nu^{(g)}_{2k''}$;
note that $\vert \su (g_o) \vert = 2u+k'$ and $\vert \su (g_e) \vert = 2v$.
By Lemma \ref{conjlengthalt}, we have 
$$
\kappa(g_o) = u , \hskip.5cm
\kappa(g_e) = v , \hskip.5cm
\text{and} \hskip.5cm
\kappa(g) =  \kappa(g_o) + \kappa(g_e) = u+v .
$$
\par

The set of conjugacy classes in $\Alt (X)$ is naturally parametrized
by pairs $(\lambda, \nu)$ of partitions such that $\nu$ 
has an even number of positive parts.
(It is important here that the set $X$ is infinite, 
otherwise some pairs correspond to \emph{two} conjugacy classes
in the alternating group).
The contribution to $C_{\Alt (X), S}(q)$ of classes of elements
such that $g = g_o$ is therefore 
$\sum_{u=0}^\infty p(u) q^u = \prod_{i=1}^\infty \frac{1}{1-q^i}$;
the contribution of classes of elements such that $g = g_e$ is 
$\sum_{v=0}^\infty p_e(v) q^v = 
\frac{1}{2} \prod_{j=1}^\infty \frac{1}{1-q^j} \, + \,
\frac{1}{2} \prod_{j=1}^\infty \frac{1}{1+q^j}$
(this uses Proposition \ref{genserpepo});
finally $C_{\Alt(X), S}$ is the product of these two contributions.
\end{proof}

\begin{rem}
(i)
Recall from Observation \ref{CongIntro} that we denote by
$$
\left( p^A(n) \right)_{n \ge 0} \, = \, (1, 1, 3, 5, 11, 18, 34, 55, 95, 150, 244, \hdots )
$$
the sequence of coefficients of the series of Proposition \ref{CforALt(N)}.
At the day of writing, this sequence
does not appear in \cite{OEIS}.

\vskip.2cm

(ii)
The sums and products in the previous proposition 
converge again for $q$ complex with $\vert q \vert < 1$.
Numerically, the roots of smallest absolute value of $C_{\Alt (\N), T^A_\N} (q)$
are simple and located at $\sim 0.67\pm 0.43i$.

\vskip.2cm

(iii)
As in the case of $C_{\Sy (X), S}(q)$, see Proposition \ref{sharpening}(a), 
it can be observed that the series $C_{\Alt (X), T^A_X}(q)$
does not depend on the cardinality of $X$,
as long as $X$ is infinite.
\end{rem}

\section{\textbf{Congruences \`a la Ramanujan for the coefficients of the series of 
Proposition \ref{CforALt(N)} }}   
\label{sectionRama}
Ramanujan, and later Watson, Atkin, Andrews, and others, have discovered 
remarkable congruence properties for the partition function,
including
\begin{eqnarray*}
p(5n+4) &\equiv& 0 \pmod{5}, \\
p(7n+5) &\equiv& 0 \pmod{7}, \\
p(11n+6) &\equiv& 0 \pmod{11}, \\
p(25n + 24) &\equiv& 0 \pmod{5^2}, \\
p(125n + 99) &\equiv& 0 \pmod{5^3}, \\
p(49n+47) &\equiv& 0 \pmod{7^2}, \\
p(121n+116) &\equiv& 0 \pmod{11^2}.
\end{eqnarray*}
See for example \cite{Hard--40},
or \cite{Bern--06} and references there.

\vskip.2cm

Consider a finite group $H$ with $M$ conjugacy class, 
an infinite set $X$,
the permutational wreath product $W = H \wr_X \Sy (X)$, 
a generating set $S$ that satisfies Condition (\ref{PCwr}),
and the corresponding conjugacy growth series
$$
C_{W,S}(q) \, = \, \prod_{k=1}^\infty \frac{1}{(1 - q^k)^M}  
\, = \, \sum_{n=0}^\infty p(n)_{(M)} q^n
$$
as in Proposition \ref{wreath}.
There is an important literature on congruence properties of the sequences
$\big( p(n)_{(M)} \big)_{n=0, 1, 2, \hdots}$ of so-called
\emph{multipartition numbers}. In particular:
\begin{equation}
\tag{Gandhi}
\label{Gandhi}
\aligned
p(5n+3)_{(2)} \, &\equiv \, 0\pmod5 ,
\\
p(11n+4)_{(8)} \, &\equiv \, 0\pmod{11} ,
\endaligned
\end{equation}
\begin{equation}
\tag{Andrews}
\label{Andrews}
p(5n+B)_{(2)} \, \equiv \, 0\pmod5 
\hskip.5cm \text{for} \hskip.2cm B \in \{2, 3, 4\} ,
\end{equation}
\begin{equation}
\tag{CDHS}
\label{CDHS}
p(25n+23)_{(2)} \, \equiv \, 0\pmod{25}  .
\end{equation}
See \cite{Gand--63}, a particular case of Theorem 1 in \cite{Andr--08},
and Fromula (1.17) in  \cite{CDHS--14}, respectively.

\vskip.2cm

Like the partition numbers $p(n)$ and the multipartition numbers $p(n)_{(M)}$, 
the coefficients of the conjugacy growth series 
$$
C_{\Alt (X), S} (q)
\, = \, 
\frac{1}{2} \prod_{j=1}^\infty \frac{1}{(1-q^j)^2} \, + \,
\frac{1}{2} \prod_{j=1}^\infty \frac{1}{1-q^{2j}}
\, = \, 
\sum_{n=0}^\infty p^A(n) q^n
$$
of Proposition \ref{CforALt(N)}
verify intriguing congruence relations,
as was recorded in Observation \ref{CongIntro}
of the Introduction.
With the notation of Appendix \ref{AppendixC}, 
the coefficients of this series  can be written as
$$
p^A(n) \, = \, \frac{1}{2} \left( p(n)_{(2)} + p(n)_{(0,1)} \right) .
$$

\begin{prop}
\label{p^Asure}
With the notation above, we have
$$
\aligned
p^A(5n+3) &\equiv 0 \pmod 5 ,
\\
p^A(10n+7) &\equiv 0 \pmod 5 ,
\\
p^A(10n+9) &\equiv 0 \pmod 5 ,
\\
p^A(25n+23) & \equiv 0 \pmod{25} .
\endaligned
$$
\end{prop}

\begin{proof}
One the one hand,
as recorded above in (\ref{Gandhi}),
it is known that $p(5n+3)_{(2)} \equiv 0 \pmod 5$
for all $n \ge 0$.
On the other hand, it follows from the definitions that
$$
p(k)_{(0,1)} \, = \, \left\{
\aligned 
p(m) \hskip.2cm &\text{if} \hskip.2cm k=2m
\\
0 \hskip.5cm &\text{if $k$ is odd}.
\endaligned
\right.
$$
Since $p(5n+4) \equiv 0 \pmod 5$ for all $n \ge 0$,
we have also $p(5n+3)_{(0,1)} \equiv 0 \pmod 5$ for all $n \ge 0$.
Hence $p^A(5n+3) = \frac{1}{2} \left( p(4n+3)_{(2)} + p(4n+3)_{(0,1)} \right)
\equiv 0 \pmod 5$ for all $n \ge 0$.
\par
Similarly, 
since $p(n)_{(0,1)} = 0$ for all odd $n$, 
the congruences for $p^A(10n+7)$ and $p^A(10n+9)$ follows from
(\ref{Andrews}), 
and for $p^A(25n+23)$ from (\ref{CDHS}).
\end{proof}

\noindent
\textbf{On the conjectured relations of Observation \ref{CongIntro}.}
For $p^A(\cdot)$,
Proposition \ref{p^Asure} contains the established part of Observation \ref{CongIntro}.
The remaining congruences of this observation follow from the congruences
\begin{align*}
&p(49n+17)_{(2)}\equiv 0\pmod 7,&p(49n+33)_{(1)}\equiv 0\pmod 7,\\
&p(49n+31)_{(2)}\equiv 0\pmod 7,&p(49n+40)_{(1)}\equiv 0\pmod 7,\\
&p(49n+38)_{(2)}\equiv 0\pmod 7,&p(49n+19)_{(1)}\equiv 0\pmod 7,\\
&p(49n+45)_{(2)}\equiv 0\pmod 7,&p(49n+47)_{(1)}\equiv 0\pmod 7,\\
&p(121n+111)_{(2)}\equiv 0\pmod{11},&p(121n+116)_{(1)}\equiv 0\pmod{11}.\\
\end{align*}
For what we know, the congruences of the left-hand side are conjectural,
with numerical evidence recorded in our Appendix \ref{AppendixC}.
The congruences on the right-hand side are all established,
and are indeed particular cases of the classical congruences $p(7n+5) \equiv 0 \pmod 7$
and $p(11n+6) \equiv 0\pmod{11}$.

\appendix

\section{\textbf{Three lemmas on symmetric and alternating groups}}
\label{AppendixA}

For reference elsewhere, 
we state here three elementary facts.
Recall from the introduction that, for $a,b \in \Sy (X)$, we agree that $ab$
denotes $b$ \emph{followed} by $a$.
The first lemma is straightforward:

\begin{lem}
\label{twocycles}
Let $X$ be a set with at least $3$ elements,
and $a,b \in \Sy (X)$ two cycles such that their supports
have exactly one element in common.
\par

Then $ab$ is a cycle 
and $\text{sup}(ab) = \text{sup}(a) \cup \text{sup}(b)$.
More precisely, if $a = (x_1, \hdots, x_r)$ and $b = (x_r, \hdots, x_{r+s-1})$,
then $ab = (x_1, \hdots, x_{r+s-1})$.
\end{lem}

The next lemma is well-known. See e.g.\ 
\cite[Lemmas 3.10.1 and 3.10.2]{GoRo--01},
where the proof of (2) is left as an exercise.

\begin{lem}
\label{arbrecycle}
Let $X$ be a non-empty set, $S$ a set of transpositions of $X$,
and $\Gamma (S)$ the transposition graph, as in Definition \ref{transpgraph}.
\begin{enumerate}[noitemsep,label=(\arabic*)]
\item\label{1DEarbrecycle}
$S$ generates $\Sy (X)$ if and only if $\Gamma (S)$ is connected.
\item\label{2DEarbrecycle}
Suppose that $X$ is finite, say of cardinality $n$,
and that $\Gamma (S)$ is a tree.
Let $s_1, s_2, \hdots, s_{n-1}$ be an enumeration of the elements of $S$.
\par
Then the product $s_1 s_2 \cdots s_{n-1}$ is a cycle of length $n$. 
\end{enumerate}
\end{lem}

\begin{proof}
(1)
Denote by $G$ the subgroup of $\Sy (X)$ generated by $S$.
\par

Suppose that $\Gamma (S)$ is not connected.
Choose a connected component of $\Gamma (S)$, 
denote by $X_1$ its vertex set, and set $X_2 = X \smallsetminus X_1$.
Then $G$ is a subgroup of the proper subgroup
$\Sy (X_1) \times \Sy (X_2)$ of $\Sy (X)$, hence $S$ does not generate $\Sy (X)$.
\par

Assume that $\Gamma (S)$ is connected.
We have to show that $G = \Sy (X)$. 
Since this is trivial when $\vert X \vert \le 2$, we assume that $\vert X \vert \ge 3$.
Let $x,y,z$ be three distinct elements in $X$;
observe that $(y,z)(x,y)(y,z) = (x,z)$.
For two distinct elements $u,v$ in $X$, it follows that $(u,v) \in G$
by induction on the length of a path connecting $u$ and $v$ in $\Gamma (S)$.
Hence $G$ contains all transpositions of elements of $X$,
and therefore $G = \Sy (X)$.

\vskip.2cm

(2)
We proceed by induction on $n$.
Note that the lemma is obvious for $n = 2$; suppose that $n > 2$,
and that the lemma holds up to $n-1$.
\par

Choose a leaf $x$ of $\Gamma (S)$. 
There is a unique $i(x) \in \{1, \hdots, n-1\}$ such that $x \in \su (s_{i(x)})$.
Upon replacing the product $s_1 \cdots s_{n-1}$ by a conjugate element, 
we can assume that $s_{i(x)} = s_{n-1}$.
By the induction hypothesis, the product $s_1 \cdots s_{n-2}$ is now a cycle $c'$ of length $n-1$.
By Lemma \ref{twocycles}, 
$s_1 \cdots s_{n-2} s_{n-1} = c' s_{n-1}$ is a cycle of length $n$.
\end{proof}

The third lemma is a cheap confirmation of the fact that most pairs of elements of $\Sy (n)$
generate either $\Alt (n)$ or $\Sy (n)$ \cite{Baba--89}.

\begin{lem}
\label{twocyclesgen}
Let $X$ be a non-empty set with at least $3$ elements,
$a,b \in \Sy (X)$ two cycles, respectively of lengths $\ell, m \ge 2$,
such that their supports have exactly one element in common
(as in Lemma \ref{twocycles}).
Let $G$ be the subgroup of $\Sy (X)$ generated by $\{a,b\}$.
\par

Then $G$ is isomorphic to the alternating group $\Alt (\ell + m - 1)$ if $\ell, m$ are both odd,
and to $\Sy (\ell + m -1)$ otherwise.
\end{lem}

\begin{proof}
Denote by $x$ the element in $\su (a) \cap \su (b)$;
set $y = a^{-1}(x)$ and $z = b^{-1}(x)$.
The commutator $a^{-1}b^{-1}ab$ is the $3$-cycle $c := (x,y,z)$.
By Lemma \ref{sgenalt} for $R^A_{\ell + 1}$,
the conjugates of $c$ by the powers of $a$ generate $\Alt (\su (a) \cup \{z\})$;
similarly the conjugates of $c$ by the powers of $b$ generate $\Alt (\{x\} \cup \su (b))$.
\par

Observe that the intersection $\Alt (\su (a) \cup \{z\}) \cap \Alt (\{x\} \cup \su (b))$ contains $c$,
and the union $\Alt (\su (a) \cup \{z\}) \cup \Alt (\{x\} \cup \su (b))$ contains
a set of $3$-cycles similar to $S^A_{\ell + m -1}$.
By Lemma \ref{sgenalt} again, this time for $S^A_{\ell+m-1}$,
the group $G$ contains $\Alt (\su (a) \cup \su(b))$, isomorphic to $\Alt (\ell + m -1)$.
\par

If $\ell$ and $m$ are both odd, every element in $G$ has an even signature,
hence $G = \Alt (\su (a) \cup \su(b)) \simeq \Alt (\ell+m-1)$.
Otherwise, $G$ is a subgroup of $\Sy (\su (a) \cup \su (b))$ 
in which $\Alt (\su (a) \cup \su(b))$ is a proper subgroup,
hence $G = \Sy (\su (a) \cup \su (b)) \simeq \Sy (\ell+m-1)$.
\end{proof}

This lemma implies for example that the set
$$
\{(0,1,2),(2,3,4),(4,5,6), \hdots, (2i, 2i+1, 2i+2), \hdots\}
$$
generates $\Alt(\N)$.
It is a proper subset of the generating set $S^A_\N$ introduced
in the beginning of Section \ref{sectionalt}.

\section{\textbf{Reminder on partitions and derangements}}
\label{AppendixB}

\subsection{The partition function}
\label{recallp(n)}   
For $n \in \N$, let $p(n)$ denote the number of partitions of $n$.
The first values are given by the table
$$
\begin{array}{c|cccccccccccccccc}
n&0&1&2&3&4&5&6&7&8&9&10&11&12&13&14&15 \\
\hline
p(n)&1&1&2&3&5&7&11&15&22&30&42&56&77&101&135&176 \\
\end{array}
$$
(more values in \cite[A000041]{OEIS}).
\par

In our context $p(n)$ is the number of conjugacy classes 
in the finite symmetric group $\Sy (n)$,
alternatively the number of conjugacy classes in $\Sy (\N)$
of elements of supports of size at most $n$.
For this reason, the partition function appears 
already in Propositions \ref{firstsample} and \ref{propS0}.
\par

It is known since Euler 
that the generating series for $p(n)$ has a product expansion
\begin{equation}
\tag{EP$_1$}
\label{EulerProd}  
\sum_{n=0}^\infty p(n) q^n 
\, = \,  \prod_{k=1}^\infty \frac{1}{1-q^k}  \hskip.1cm .
\end{equation}
See \cite[Caput XVI]{Eule--48},
as well as, for example, \cite[Section 19.3]{HaWr--79}.
The equality can be viewed either between formal expressions,
or between absolutely converging sum and product for $q \in \C$ with $\vert q \vert < 1$.
\par

There is an asymptotic formula for $n \to \infty$
$$
\aligned
p(n) \, &= \, \frac{1}{4 \sqrt 3 \left(n - \frac{1}{24}\right)} \hskip.1cm
                 \exp \left( \pi \sqrt{ \frac{2}{3}  \left(n - \frac{1}{24}\right) } \hskip.1cm \right)
\\               
         \, &+ \, O \left( \frac{1}{\left(n - \frac{1}{24}\right)^{3/2}} \hskip.1cm
         \exp \left( \pi \sqrt{ \frac{2}{3}  \left(n - \frac{1}{24}\right) } \hskip.1cm \right) \right)
\endaligned
$$
due to Hardy and Ramanujan
\cite[Formula (1.41)]{HaRa--18}.
For this and more on $p(n)$ when $n \to \infty$,
see e.g.\ \cite[Chapter VII]{Chan--70} and \cite[Chapters VI and VIII]{Hard--40}.
This shows in particiular that the sequence $(p(n))_{n \ge 0}$ has \textbf{intermediate growth},
i.e.\ that its growth is superpolynomial and subexponential.

\subsection{Partitions with $k$ parts}
\label{recallpk(n)}
For $n,k \in \N$, we denote by $p_k(n)$
the number of partitions of $n$ in exactly $k$ positive parts,
equivalently the number of partitions of $n$ with largest part $k$,
equivalently the number of partitions of $n-k$ in $k$ non-negative parts.
Whenever needed, we set $p_k(n) = 0$ for all $n \in \N$ and $k < 0$.
Numbers $p_k(\cdot)$ appear 
in connection with finite symmetric groups, 
in Propositions \ref{firstsample},
\ref{propS0}, \ref{rodrigues}, and \ref{growthpolSym(n)S0}.
\par

We have classically
$$
\aligned
p_0(0) \, &= \, 1 \hskip.5cm \text{and} \hskip.5cm 
   p_0(n) \, = \, 0 \hskip.5cm \text{for all} \hskip.2cm n \ge 1,
\\
p_1(0) \, &= \, 0 \hskip.5cm \text{and} \hskip.5cm
p_1(n) \, = \, 1 \hskip.5cm \text{for all} \hskip.2cm n \ge 1,
\\
p_2(n) \, &= \, \lfloor n/2 \rfloor \hskip.5cm \text{for all} \hskip.2cm n \ge 0,
\\
p_3(n) \, &= \, \lfloor \frac{1}{12}(n^2+6) \rfloor \hskip.5cm \text{for all}
   \hskip.2cm n \ge 0 \hskip.2cm \text{\cite[A069905]{OEIS}},
\\
\hdots \, & \hskip1cm \hdots
\\
p_{n-2}(n) \, &= \, 2 \hskip.5cm \text{for all} \hskip.2cm n \ge 4,
\\
p_{n-1}(n) \, &= \, p_n(n) \, = \, 1 \hskip.5cm \text{for all} \hskip.2cm n \ge 2,
\\
p_k(n) \, &= \, 0  \hskip.5cm \text{for all} \hskip.2cm k > n  \ge 0 ,
\\
p_k(n) \, &= \, p_k(n-k) + p_{k-1}(n-1)
\hskip.5cm \text{for all} \hskip.2cm n \ge k \ge 1 ,
\\
\sum_{k=0}^n& p_k(n) \, = \, \sum_{k=1}^n p_k(n) \, = \, p(n) 
\hskip.5cm \text{for all} \hskip.2cm n \ge 1
\hskip.1cm ,
\endaligned
$$
and the generating function
\begin{equation}
\tag{EP$_2$}
\label{Euler312}
\sum_{n \ge 0} p_k(n) q^n \, = \, q^k \prod_{i=1}^k \frac{1}{1-q^i} 
\hskip.5cm \text{for all} \hskip.2cm k \ge 0 
\hskip.1cm .
\end{equation}
(Observe that $\sum_{n \ge 0} p_k(n) q^n = \sum_{n \ge k} p_k(n) q^n$.)
Up to the notation, Equality (\ref{Euler312})
is contained in Number 312 of \cite[Caput XVI]{Eule--48}.
\par

Moreover, if $P(n,t) := \sum_{k=0}^n p_k(n)t^k$, then
\begin{equation}
\tag{EP$_3$}
\label{Euler304}
\sum_{n=0}^\infty P(n,t) q^n \, = \, \prod_{j=1}^\infty \frac{1}{1 - tq^j} 
\hskip.1cm .
\end{equation}
This appears in Number 304 of \cite[Caput XVI]{Eule--48},
and is used in the proof of our Proposition \ref{genserpepo}.
\par

For $n,\ell \in \N$ with $n \le 2\ell$, every partition of $n-\ell$ has at most $\ell$ parts.
Thus every partition of $n-\ell$ can be obtained 
from a unique partition of $n$ in $\ell$ parts
by substracting $1$ from each part. Consequently
\begin{equation}
\tag{EP$_4$}
\label{Euler?}
p_\ell(n) \, = \, p(n-\ell)
\hskip.5cm \text{for integers} \hskip.2cm n,\ell 
\hskip.2cm \text{such that} \hskip.2cm 0 \le \ell \le n \le 2\ell
\hskip.1cm ,
\end{equation}
or, setting $k = n-\ell$, 
\begin{equation}
\tag{EP$^{'}_4$}
\label{Euler??}
p_{n-k}(n) \, = \, p(k)
\hskip.5cm \text{for integers} \hskip.2cm n, k 
\hskip.2cm \text{such that} \hskip.2cm k \ge 0 \hskip.2cm \text{and} \hskip.2cm 2k \le n 
\hskip.1cm .
\end{equation}

\par

The double sequence $\left( p_k(n) \right)_{n \ge 0, \hskip.1cm 0 \le k \le n}$
gives rise to a generalized Pascal triangle
of which the first rows are:

\tiny
\begin{equation}
\label{PTp}
\tag{PTp}
\begin{array}{ccccccc}
p_0(0) & & & & & &
\\ 
p_0(1) & p_1(1) & & & & &
\\
p_0(2) & p_1(2) & p_2(2) & & & &
\\
p_0(3) & p_1(3) & p_2(3) & p_3(3) & & &
\\
p_0(4) & p_1(4) & p_2(4) & p_3(4) & p_4(4)& &
\\
p_0(5) & p_1(5) & p_2(5) & p_3(5) & p_4(5) & p_5(5) &
\\
p_0(6) & p_1(6) & p_2(6) & p_3(6) & p_4(6) & p_5(6)  & p_6(6)
\\
p_0(7) & p_1(7) & p_2(7) & p_3(7) & p_4(7) & p_5(7)  & 
\hdots
\end{array}
\, = \, 
\begin{array}{ccccccc}
1 & & & & & &
\\ 
0 & 1 & & & & &
\\
0 & 1 & 1 & & & &
\\
0 & 1 & 1 & 1 & & &
\\
0 & 1 & 2 & 1 & 1 & &
\\
0 & 1 & 2 & 2 & 1 & 1 &
\\
0 & 1 & 3 & 3 & 2 & 1 & 1
\\
0 & 1 & 3 & 4 & 3 & 2 &
 \hdots
\end{array}
\end{equation}
\normalsize

\subsection{Partitions with even or odd numbers of parts}
\label{recallpepo(n)}

We denote by $p_e(n)$, respectively $p_o(n)$, the number of partitions
of a non-negative integer $n$ involving an even, respectively odd, number of non-zero parts. 
Working with conjugate partitions, we see that
$p_e(n)$, respectively $p_o(n)$, is equivalently given by the number
of partitions of $n$ having an even largest part, respectively an odd largest part.
We have the trivial identity $p(n)=p_e(n)+p_o(n)$. 
These numbers $p_e(n)$ appear in Proposition \ref{CforALt(N)}.
Their values for $n \le 15$ are given by
$$
\begin{array}{c|cccccccccccccccc}
n&0&1&2&3&4&5&6&7&8&9&10&11&12&13&14&15\\
\hline
p_e(n)&1&0&1&1&3&3&6&7&12&14&22&27&40&49&69&86\\
p_o(n)&0&1&1&2&2&4&5&8&10&16&20&29&37&52&66&90\\
p(n)&1&1&2&3&5&7&11&15&22&30&42&56&77&101&135&176\\
\end{array}
$$
see A027187 and A027193 of \cite{OEIS}.
 
\begin{prop} 
\label{genserpepo}
(1)
The generating series of the sequence $p_e(n)$ is 
\begin{eqnarray*}
\sum_{n=0}^\infty p_e(n)q^n \, &=& \,  
\sum_{k=0}^\infty q^{2k}\prod_{j=1}^{2k} \frac{1}{1-q^j}
\\
\, &=& \, 
\frac{1}{2}\left(\prod_{j=1}^\infty\frac{1}{1-q^j}+
\prod_{j=1}^\infty\frac{1}{1+q^j}\right)
\\
\, &=& \,
\prod_{j=1}^\infty \frac{1}{1-q^j} \sum_{m=0}^\infty (-q)^{m^2}.
\end{eqnarray*}
\par

(2)
The generating series of the sequence $p_o(n)$ is 
\begin{eqnarray*}
\sum_{n=0}^\infty p_o(n)q^n\, &=& \, 
\sum_{k=0}^\infty q^{2k+1}\prod_{j=1}^{2k+1}
\frac{1}{1-q^j}
\\
\, &=& \, 
\frac{1}{2}\left(\prod_{j=1}^\infty\frac{1}{1-q^j}
-\prod_{j=1}^\infty\frac{1}{1+q^j}\right)
\\
\, &=& \, 
-\prod_{j=1}^\infty \frac{1}{1-q^j} \sum_{m=1}^\infty (-q)^{m^2}.
\end{eqnarray*}
\end{prop}

\begin{proof}
(1)
Using (\ref{Euler312}), we have
$$
\aligned  
\sum_{n=0}^\infty p_e(n)q^n
\, &= \, 
\sum_{n=0}^\infty \sum_{k=0}^{\lfloor n/2 \rfloor} p_{2k}(n) q^n 
\, = \, 
\sum_{k=0}^\infty \sum_{n=0}^\infty p_{2k}(n) q^n
\\
\, &= \, \sum_{k=0}^\infty q^{2k} \prod_{j=1}^{2k} \frac{1}{1-q^j} .
\endaligned
$$
Also, if $P(n,t) := \sum_{k=0}^n p_k(n)t^k$ as in (\ref{Euler304}), then
$$
\aligned
\sum_{n=0}^\infty p_e(n)q^n \, &= \, 
\frac{1}{2} \left(
\sum_{n=0}^\infty P(n,1) q^n + \sum_{n=0}^\infty P(n, -1) q^n
\right)
\\
\, &= \, 
\frac{1}{2} \left(
\prod_{j=1}^\infty \frac{1}{1-q^j}  + \prod_{j=1}^\infty \frac{1}{1 + q^j}
\right) .
\endaligned
$$
For the third equality in (1), 
one way is to refer to \cite{Fine--88}:
see there Equation (7.324), Page 6, and also Example 7, Page 39.
\par

The proof of (2) is similar.
\par

\vskip.2cm

Here is an alternative to citing \cite{Fine--88}. We have
$$
\aligned
&
\sum_{n=0}^\infty \left( p_e(n) - p_o(n) \right) q^n 
\, = \, 2 \sum_{n=0}^\infty  p_e(n) q^n 
\, - \, \sum_{n=0}^\infty  p(n) q^n 
\, = \, \prod_{j=1}^\infty \frac{1}{1+q^j} 
\\
& \hskip.5cm
\, = \, \prod_{j=1}^\infty \frac{1-q^j}{1-q^{2j}} 
\, = \, \prod_{j=1}^\infty (1 - q^{2j-1})
\, = \, \prod_{j=1}^\infty \frac{1}{1 - q^j} \hskip.1cm
\prod_{k=1}^\infty (1 - q^{2k-1})^2 (1 - q^{2k}) \hskip.1cm .
\endaligned
$$
The \emph{Jacobi triple product identity} reads
$$
\prod_{k=1}^\infty \Big(1 - x^{2k}\Big) \Big(1 + x^{2k-1}y^2\Big) \Big(1 + \frac{x^{2k-1}}{y^2} \Big)
\, = \, 
\sum_{n=-\infty}^\infty x^{n^2}y^{2n} 
$$
(see e.g.\ 
\cite[Theorem 352]{HaWr--79}).
For $x = q$ and $y = \sqrt{-1}$ it reduces to 
$$
\prod_{k=1}^\infty (1 - q^{2k-1})^2 (1 - q^{2k}) \, = \, 
\sum_{n=-\infty}^\infty (-q)^{n^2} \hskip.1cm ,
$$
hence
$$
\sum_{n=0}^\infty \left( p_e(n) - p_o(n) \right) q^n 
\, = \, 
\prod_{j=1}^\infty \frac{1}{1-q^j} \hskip.1cm 
\sum_{n=-\infty}^\infty (-q)^{n^2} \hskip.1cm .
$$
Finally:
$$
\aligned
\sum_{n=0}^\infty p_e(n) q^n \, &= \,
\frac{1}{2} \sum_{n=0}^\infty \left( p_e(n) - p_o(n) \right) q^n
+
\frac{1}{2} \sum_{n=0}^\infty \left( p_e(n) + p_o(n) \right) q^n
\\
\, &= \,
\frac{1}{2} \prod_{j=1}^\infty \frac{1}{1-q^j} \hskip.1cm \sum_{n=-\infty}^\infty (-q)^{n^2}
\, + \, \frac{1}{2} \prod_{j=1}^\infty \frac{1}{1-q^j}
\\
\, &= \,
\prod_{j=1}^\infty \frac{1}{1-q^j} \hskip.1cm \sum_{n=0}^\infty (-q)^{n^2} \hskip.1cm ,
\endaligned
$$
as was to be shown.
\end{proof}

\begin{obs}
We have
$$
\Big( \sum_{n=0}^\infty p_e(n)q^n \Big)^2 - \Big( \sum_{n=0}^\infty p_o(n)q^n \Big)^2
\, = \, 
\sum_{n=0}^\infty p(n) q^{2n} \, = \, \prod_{j=1}^\infty \frac{1}{1 - q^{2j}} .
$$
\end{obs}

\begin{proof}
The left-hand side can be written as
$$
\aligned
&
\frac{1}{4}\left(\prod_{j=1}^\infty\frac{1}{1-q^j} +
\prod_{j=1}^\infty\frac{1}{1+q^j}\right)^2 \, - \, 
\frac{1}{4}\left(\prod_{j=1}^\infty\frac{1}{1-q^j} -
\prod_{j=1}^\infty\frac{1}{1+q^j}\right)^2 
\\
& \hskip1cm \, = \,
\prod_{j=1}^\infty \frac{1}{1 - q^{2j}} \hskip.1cm ,
\endaligned
$$
and the claim follows. 
\end{proof}

\subsection{Derangements that are products of $k$ cycles}
\label{recalldk(n)}

A \textbf{derangement} is a fixed point free permutation.
For $n,k \in \N$, denote by $d_k(n)$ the number of derangements
of $\{1, 2, \hdots, n\}$ that are products of $k$ disjoint cycles.
These numbers appear in Remark \ref{badwritingforC}
and Proposition \ref{growthpolSym(n)S0}.

\begin{lem}
\label{dnk}
With the notation above, we have
\begin{enumerate}[noitemsep,label=(\roman*)]
\item\label{iDEdnk}
$d_0(0) = 1$ ;
\item\label{iiDEdnk}
$d_k(1) = 0$ for all $k \in \N$ ;
\item\label{iiiDEdnk}
$d_k(n) \, = \, 0$ for all $n,k \in \N$ with  $k=0$ or  $2k > n$ ;
\end{enumerate}
For all $n \ge 2$ and $k \ge 1$, we have
\begin{enumerate}[noitemsep,label=(\roman*)]
\addtocounter{enumi}{3}
\item\label{ivDEdnk}
$d_k(n) = (n-1)\big(d_k(n-1) + d_{k-1}(n-2) \big)$ ;
\item\label{vDEdnk}
$d_k(n) = \sum_{a=2}^n \binom{n-1}{a-1} (a-1)! \hskip.1cm d_{k-1}(n-a)$.
\end{enumerate}
\end{lem}

\begin{proof}
Claims \ref{iDEdnk} to \ref{iiiDEdnk} are obvious.
\par

For \ref{ivDEdnk}, consider a derangement 
$g$ of $\{1, \hdots, n\}$ product of $k$ cycles.
\par

Either $n$ is in the support of a cycle $(x_1, \hdots, x_{\ell - 1}, n)$ of length at least $3$.
Replacing it by the cycle $(x_1, \hdots, x_{\ell-1})$ produces
a derangement of $\{1, \hdots, n-1\}$ product of $k$ cycles,
and each of the latter is obtained $n-1$ times in this way.
This explains the contribution $(n-1)d_k(n-1)$ of the right-hand side.
\par
Or $n$ is in the support of a transposition, say $(i,n)$ 
with  $i \in \{1, \hdots, n-1\}$,
so that $g$ is the product of $(i,n)$
with a derangement $h$ of $\{1, \hdots, n-1\} \smallsetminus \{i\}$
product of $k-1$ cycles.
For each of the $n-1$ possible values of $i$,
there are $d_{k-1}(n-2)$ such permutations $h$,
and this explains the contribution $(n-1)d_{k-1}(n-2)$.
\par

For \ref{vDEdnk}, a permutation contributing to $d_k(n)$
is the product of a cycle $c$ of length $a \ge 2$, with $n \in \su (c)$,  
and there are $\binom{n-1}{a-1} (a-1)!$ such cycles,
with a derangement of
$\{1, \hdots, n\} \smallsetminus \su (c)$ 
which is a product of $k-1$ cycles. 
\end{proof}

\begin{rem}
(i)
The double sequence $\left( d_k(n) \right)_{n \ge 0, \hskip.1cm 0 \le k \le n}$
gives rise to a generalized Pascal triangle
of which the first rows are:
\tiny
\begin{equation}
\label{PTd}
\tag{PTd}
\begin{array}{cccccc}
d_0(0) & & & & &
\\ 
d_0(1) & d_1(1) & & & &
\\
d_0(2) & d_1(2) & d_2(2) & & &
\\
d_0(3) & d_1(3) & d_2(3) & d_3(3) & &
\\
d_0(4) & d_1(4) & d_2(4) & d_3(4) & d_4(4)&
\\
d_0(5) & d_1(5) & d_2(5) & d_3(5) & d_4(5) & d_5(5)
\\
d_0(6) & d_1(6) & d_2(6) & d_3(6) & d_4(6) & \hdots
\end{array}
\, = \, 
\begin{array}{cccccc}
1 & & & & &
\\ 
0 & 0 & & & &
\\
0 & 1 & 0 & & &
\\
0 & 2 & 0 & 0 & &
\\
0 & 6 & 3 & 0 & 0 &
\\
0 & 24 & 20 & 0 & 0 & 0
\\
0 & 120 & 130 & 15 & 0 & 
 \hdots
\end{array}
\end{equation}
\normalsize

\vskip.2cm 

(ii)
Besides the relations of Lemma \ref{dnk}, we have also
\begin{equation}
\label{Sigmad}
\tag{$\Sigma d$}
\sum_{m=0}^n \binom{n}{m} \sum_{k=0}^m  d_k(m) \, = \, n! 
\hskip.5cm \text{for all} \hskip.2cm n \in \N ,
\end{equation}
which is useful to check numerical values.
Indeed,
each of the $n!$ permutations $g$ of $\{1, \hdots, n\}$
induces a derangement of $\su (g)$.
For $m \in \{0, 1, \hdots, n\}$, there are $\binom{n}{m}$ subsets 
of $\{1, \hdots, n\}$ of size $m$.
Since there are $\sum_{k=0}^m d_k(m)$ derangements
of each of these subsets, we obtain the left-hand side.
Relation (\ref{Sigmad}) reduces to $d_0(0) = 1$ for $n=0$,
and to $d_0(0) + d_0(1) + d_1(1) = 1 + 0 + 0 = 1$ for $n=1$.
Otherwise, it can be written
\begin{equation}
\label{Sigmad'}
\tag{$\Sigma d'$}
1 + \sum_{m=2}^n \binom{n}{m} \sum_{k=1}^{\lfloor m/2 \rfloor} d_k(m) \, = \,  n!
\hskip.5cm \text{for all} \hskip.2cm n \ge 2 .
\end{equation}
\par

The sum $d(m) := \sum_{k=0}^m  d_k(m) = \sum_{k=0}^{\lfloor m/2 \rfloor} d_k(m)$ 
is the number of derangements of $m$ objects,
and there is a classical formula:
$$
d(m) \, = \, 
\sum_{k=0}^m  d_k(m) \, = \,
m! \hskip.1cm \left( 1 - \frac{1}{1!} + \frac{1}{2!} - \frac{1}{3!} + \cdots + (-1)^m \frac{1}{m!} \right) 
$$
for all $m \ge 0$; it follows that we have the relations
$$
\aligned
d(m) \, &= \, md(m-1) + (-1)^m
\hskip.5cm \text{for all} \hskip.2cm m \ge 1 \hskip.1cm ,
\\
d(m) \, &= \, (m-1) \big( d(m-1) + d(m-2) \big)
\hskip.5cm \text{for all} \hskip.2cm m \ge 2 \hskip.1cm ;
\endaligned
$$
see e.g.\ \cite[Example 2.2.1]{Stan--97}.
The sequence 
$$
(d(m))_{m \ge 0} 
\, = \,  (1, 0, 1, 2, 9, 44, 265, 1854, 14833, 133496, 1334961, ...)
$$ 
is A000166 in \cite{OEIS}.

\vskip.2cm

(iii) 
Numbers $d_k(n)$ have some flavour of Stirling numbers.
For $n, k \in \N$ with $0 \le k \le n$, recall that the 
\textbf{unsigned Stirling number of the fist kind} ${n \brack k}$
counts the number of ways to arrange $n$ objects into $k$ cycles
(here, cycles of length $1$ are included, unlike elsewhere in this article,
and this is why entries in (\ref{PTd}) are smaller or equal than entries in 
(\ref{PTStirlingNumbersFirstKind}).
When $n \ge 1$, we have
${n \brack k} = (n-1) {n-1 \brack k} + {n-1 \brack k-1}$.
See for example \cite[Page 245]{GrKP--89} and \cite[A132393]{OEIS}.
The generalized Pascal triangle for 
$\left( {n \brack k} \right)_{n \ge 0, \hskip.1cm 0 \le k \le n}$
is
\begin{equation}
\label{PTStirlingNumbersFirstKind}
\tag{PTStir}
\begin{array}{ccccccc}
1 & & & & & &
\\ 
0 & 1 & & & & &
\\
0 & 1 & 1 & & & &
\\
0 & 2 & 3 & 1 & & &
\\
0 & 6 & 11 & 6 & 1 & &
\\
0 & 24 & 50 & 35 & 10 & 1 &
\\
0 & 120 & 274 & 225 & 85 & 15 & 1
\end{array}
\end{equation}
Note that we have
$$
{n \brack k}
\, = \,
\sum_{j=0}^k \binom{n}{j} d_{k-j}(n-j) \hskip.1cm .
$$
Indeed, in the right-hand side,
the term with a given value of $j$
counts the number of contributions to ${n \brack k}$ with $j$ fixed points.
\end{rem}

\section{Generalized Ramanujan congruences}
\label{AppendixC}

This appendix is  partly experimental.
It grew out of our desire to understand the reasons
for the congruences for the numbers $p^A(n)$ described 
in Observation \ref{CongIntro} and Section \ref{sectionRama}.

\subsection{Definitions}

\begin{defn}
Given a sequence $\mathbf{e} = (e_1, e_2, e_3, \hdots) \in \Z^{(1, 2, 3, \hdots )}$
of integers with $e_d = 0$ for $d$ large enough,
the corresponding \textbf{generalized partition numbers} $p(n)_{\mathbf e}$
are the coefficients of the power series
\begin{equation}
\label{defgen}
\aligned
\sum_{n=0}^\infty p(n)_{\mathbf e} q^n \, &= \
\prod_{n=1}^\infty \prod_{d=1}^\infty \frac{1}{ (1-q^{dn})^{e_d} } 
\\
\, &= \,
\prod_{n=1}^\infty \frac{1}{ (1-q^n)^{e_1} (1-q^{2n})^{e_2} (1 - q^{3n})^{e_3} \, \cdots } \hskip.1cm .
\endaligned
\end{equation}
\end{defn}

\begin{rem}
As a shorthand,
we also write a sequence $\mathbf e$ as above as $(e_1, e_2, \hdots, e_k)$
when $e_k \ne 0$ and $e_d = 0$ for all $d \ge k+1$. 
For example:
\begin{equation}
\label{defgen2}
\sum_{n=0}^\infty p(n)_{(0,3)} q^n \, = \
\prod_{n=1}^\infty \frac{1}{ (1-q^{2n})^3 } \hskip.1cm .
\end{equation}
For a sequence of the form $(e_1, \hdots, e_j, 0, \hdots, 0, e_k)$
with $e_k \ne 0$, and $e_d = 0$ when $j < d < k$ or $d > k$,
we also write $(e_1, \hdots, e_j, (e_k)_k)$.
For example:
\begin{equation}
\label{defgen3}
\sum_{n=0}^\infty p(n)_{(0, 1, 2_8)} q^n \, = \, 
\prod_{n=1}^\infty \frac{1}{  (1-q^{2n}) (1 - q^{8n})^2 } \hskip.1cm .
\end{equation}
\par

For a positive integer $M$, 
the numbers $p(n)_{(M)}$ arising as coefficients 
of the series defined by $\prod_{n=1}^\infty \frac{1}{(1-q^n)^M }$
are called \textbf{multi-partition numbers} in the literature,
since $p(n)_{(M)}$ counts the number of ways of writing $n$ 
as a sum of parts, each coloured in one of $M$ colours.
More generally, for $\mathbf{e} \in \Z^{(1, 2, 3, \hdots )}$ as above,
$p(n)_{\mathbf{e}}$ can be interpreted as multi-partition numbers
which constraints on the parts; 
for example, the coefficient $p(n)_{(0, 1, 2_8)}$ of (\ref{defgen3})
counts the number of partitions of the form 
$$
n = \lambda_1 + \cdots + \lambda_i+
\mu_1 + \cdots + \mu_j +
\nu_1 + \cdots + \nu_k
$$
where
$$
\aligned
\lambda_1 \ge \cdots \ge \lambda_i \ge 1
&\hskip.2cm \text{and $\lambda_1, \hdots, \lambda_i$ are even,}
\\
\mu_1 \ge \cdots \ge \mu_j \ge 1
&\hskip.2cm \text{and $\mu_1, \hdots, \mu_j$ are multiples of $8$,}
\\
\nu_1 \ge \cdots \ge \nu_k \ge 1
&\hskip.2cm \text{and $\nu_1, \hdots, \nu_k$ are multiples of $8$.}
\endaligned
$$
\end{rem}

\begin{defn}
\label{defgRc}
A \textbf{generalized Ramanujan congruence} is 
\begin{enumerate}[noitemsep]
\item[-]
a sequence $\mathbf{e} = (e_1, e_2, e_3, \hdots) \in \Z^{(1, 2, 3, \hdots )}$ as above,
\item[-]
an arithmetic progression $(An+B)_{n \ge 0}$ 
with $A \ge 2$ and $1 \le B \le A-1$
\item[-]
a prime power $\ell^f$, with $\ell$ prime and $f \ge 1$, 
\end{enumerate}
such that
\begin{equation}
\label{RamCongDef}
p(An+B)_{\textbf{e}} \, \equiv \, 0 \pmod {\ell^f}
\hskip.5cm \text{for all} \hskip.2cm n \ge 0 .
\end{equation}
\end{defn}

\begin{obs}
\label{obsgenRamcong}
(1)
Let $p(An+B)_{\textbf{e}} \equiv 0 \pmod {\ell^f}$ 
be a generalized Ramanujan congruence as above,
and let $m \ge 2$.
Define a sequence $\textbf{e}'$ by $e'_d = e_{d/m}$ if $m$ divides $d$
and $e'_d = 0$ otherwise. Then we have
\begin{equation}
\label{nonprim}
\aligned
p(mAn+mB)_{\textbf{e}'} \, &\equiv \, 0 \pmod {\ell^f}
\hskip1cm \text{for all} \hskip.2cm n \ge 0 ,
\\
p(mn+B')_{\textbf{e}'} \, &= \, 0 \hskip1cm \text{for all} \hskip.2cm n \ge 0
\hskip.2cm \text{and} \hskip.2cm B' \in \{1, 2, \hdots, m-1\}.
\endaligned
\end{equation}
Observe that the integers of the \textbf{support} $\{d \ge 1 \mid e'_d \ne 0 \}$
 of $\textbf{e}'$ have a common divisor $m \ge 2$.
\par

A generalized Ramanujan congruence is \textbf{primitive} if 
the integers in its support are coprime.
All examples of generalized Ramanujan congruences appearing below are primitive.

\vskip.2cm

(2)
In lists of examples involving congruences modulo $\ell$ (and not $\ell^f$ with $f \ge 2$), 
we write shortly $p(\ell n+B)_{\textbf{e}}$
for $p(\ell n+B)_{\textbf{e}} \equiv 0 \pmod \ell$.
The Ramanujan congruences of this sort in Section \ref{sectionRama}  
can therefore be written
$$
p(5n+4)_{(1)}, \hskip.2cm
p(7n+5)_{(1)}, \hskip.2cm
p(11n+6)_{(1)}, \hskip.2cm
p(5n+B)_{(2)}, \hskip.2cm
p(11n+4)_{(2)} .
$$
(With $B \in \{2,3,4\}$.)
Moreover, we also write    
$$
p(\ell n+B)_{\mathbf e,\hskip.1cm \mathbf e',\hskip.1cm \hdots,\hskip.1cm  \mathbf e''}
$$
as a shorthand for 
$p(\ell n+B)_{\mathbf e}$,
$p(\ell n+B)_{\mathbf e'}$, $\hdots$,
$p(\ell n+B)_{\mathbf e''}$.
\par
This shorthand notation will be used systematically in the lists 
of Subsections \ref{subsectionp=3} to \ref{subsectionp=13}.

\vskip.2cm

(3)
When we consider below generalized Ramanujan congruence involving a prime $\ell$
(and not a prime power $\ell^f$ with $f \ge 2$), 
it suffices to consider sequences $\mathbf{e} = (e_1, e_2, e_3, \hdots)$
with $0 \le e_d \le \ell-1$ for all $d \ge 0$.
This is a corollary of the following standard proposition, 
for which we did not find a convenient reference.
\end{obs}

\begin{prop}
\label{Frobenius}
Let $\ell$ be a prime, 
$S(q) = \sum_{n=0}^\infty s_n q^n, T(q) = \sum_{n=0}^\infty t_n q^n$
two power series in $\Z [[q]]$,
and $(pn+B)_{n \ge 0}$ an arithmetic progression of common difference $\ell$
and first term $B \ge 1$ not divisible by $\ell$. Set
$$
U(q) \, = \, S(q) \big( T(q) \big)^\ell \, = \, \sum_{n=0}^\infty u_n q^n .
$$
Assume that $s_{\ell n+B} \equiv 0 \pmod \ell$ for all $n \ge 0$.
\par
Then $u_{\ell n+B} \equiv 0 \pmod \ell$ for all $n \ge 0$.
\end{prop}

\begin{proof}
For the binomial coefficients, we have the well-known congruences
$$
\binom{\ell}{j} \equiv 0 \pmod \ell
\hskip.5cm \text{for all} \hskip.2cm j \ge 0
\hskip.2cm \text{with} \hskip.2cm j \not\equiv 0 \pmod \ell .
$$
Hence the power series
$\big( T(q) \big)^\ell = \sum_{n=0}^\infty t'_n q^n$ 
and $T(q^\ell) = \sum_{n=0}^\infty t_n q^{\ell n}$
have coefficients that are congruent modulo $\ell$;
in particular, $t'_n \equiv 0 \pmod \ell$ for all $n \ge 0$ with $n \not\equiv 0 \pmod \ell$.
\par
In particular, if
$s_{\ell n+B} \equiv 0 \pmod \ell$ for all $n \ge 0$,
then $u_{\ell n+B} \equiv 0 \pmod \ell$ for all $n \ge 0$.
\end{proof}

\begin{cor}
\label{corFrob}
Consider a sequence $\mathbf e = (e_1, e_2, e_3, \hdots) \in \Z^{(1, 2, 3, \hdots)}$,
an arithmetic progression $(An+B)_{n \ge 0}$ with $A \ge 2$ and $1 \le B \le A-1$,
a prime $\ell$, and another sequence $\mathbf{e}' = (e'_1, e'_2, e'_3, \hdots) \in \Z^{(1, 2, 3, \hdots)}$.
Assume that $e'_d \equiv e_d \pmod \ell$ for all $d \ge 0$.
\par

If $p(An+B)_{\mathbf e} \equiv 0 \pmod \ell$ for all $n \ge 0$ (as in Definition \ref{defgRc}), 
then $p(An + B)_{\mathbf{e}'} \equiv 0 \pmod \ell$ for all $n \ge 0$.
\end{cor}

We now proceed to indicate a list of examples of generalized Ramanujan congruences.
Except for a few exceptions, they are CONJECTURAL.
In each case, they have been checked numerically, for $p(n)_{\mathbf e}$
with $n \le 5000$.
\par
We use the shorthand notation explained in Remark \ref{obsgenRamcong}(2).

\subsection{Some examples of the form $p(3n+B)_{\mathbf e} \equiv 0 \pmod 3$}
\label{subsectionp=3}   

\begin{align*}
&p(3n+2)_{(1,1),\hskip.1cm (2,1,0,2),\hskip.1cm (2,1,0,1,2,1_{10},1_{20}),\hskip.1cm 
(1,1,0,2,1,1_{10},2_{20})} \hskip.1cm .
\end{align*}

\subsection{Some examples of the form $p(5n+B)_{\mathbf e}\equiv 0\pmod 5$}
\label{subsectionp=5}
For $\ell=5$ and when $e_d = 0$ for all $d \ge 3$, 
we find the Ramanujan congruences\begin{align*}
&p(5n+2)_{(2),\hskip.1cm (3,1),\hskip.1cm (1,3)}, \hskip.5cm
p(5n+3)_{(2),\hskip.1cm (4),\hskip.1cm (3,1)}, \hskip.5cm
p(5n+4)_{(1),\hskip.1cm (2),\hskip.1cm (4),\hskip.1cm (2,2),\hskip.1cm (1,3)} \hskip.1cm .
\end{align*}
When $e_d = 0$ for all $d$ not dividing $4$, we find moreover the Ramanujan congruences
\begin{align*}
&p(5n+2)_{(2,0,0,2),\hskip.1cm (3,1,0,2),\hskip.1cm 
(3,1,0,3),\hskip.1cm (2,0,0,4),\hskip.1cm (4,1,0,4)} \hskip.1cm , \\
&p(5n+3)_{(1,2,0,1),\hskip.1cm (2,0,0,2),\hskip.1cm (4,0,0,2),\hskip.1cm 
(3,1,0,3),\hskip.1cm (1,2,0,3)} \hskip.1cm , \\
&p(5n+4)_{(1,2,0,1),\hskip.1cm (3,2,0,1),\hskip.1cm (2,1,0,3),\hskip.1cm 
(3,1,0,3),\hskip.1cm (3,3,0,3),\hskip.1cm (4,1,0,4),\hskip.1cm (4,3,0,4)} \hskip.1cm .
\end{align*}
When $e_d = 0$ for all $d$ not dividing $6$, we find moreover  
\begin{align*}
&p(5n+1)_{(0,2,2),\hskip.1cm (0,4,2),\hskip.1cm (0,2,3,0,0,1)} \hskip.1cm , \\
&p(5n+2)_{(1,3,2),\hskip.1cm (1,3,4,0,0,1),\hskip.1cm (4,1,1,0,0,3),\hskip.1cm 
(4,1,3,0,0,3),\hskip.1cm (3,1,1,0,0,4),\hskip.1cm (3,1,3,0,0,4)} \hskip.1cm , \\
&p(5n+3)_{(1,1,1,0,0,1),\hskip.1cm (1,4,3,0,0,1),\hskip.1cm 
(1,3,4,0,0,1),\hskip.1cm (3,3,4,0,0,1)} \hskip.1cm , \\
&p(5n+3)_{(3,1,0,0,0,2),\hskip.1cm (2,3,4,0,0,2),\hskip.1cm 
(4,2,2,0,0,3),\hskip.1cm (3,2,2,0,0,4)} \hskip.1cm , \\
&p(5n+4)_{(0,2,2),\hskip.1cm (0,2,4),\hskip.1cm (1,4,3,0,0,1),\hskip.1cm 
(3,4,3,0,0,1),\hskip.1cm (2,4,3,0,0,2)}\hskip.1cm , \\
&p(5n+4)_{(4,1,1,0,0,3),\hskip.1cm (4,3,1,0,0,3),\hskip.1cm (1,4,3,0,0,3),\hskip.1cm 
(3,1,1,0,0,4),\hskip.1cm (3,3,1,0,0,4)} \hskip.1cm .
\end{align*}
When $e_d = 0$ for all $d$ not dividing $8$, we find moreover
\begin{align*}
&p(5n+2)_{(2,2_8),\hskip.1cm (1,3,2_8),\hskip.1cm 
(3,1,0,3,2_8),\hskip.1cm (4,1,0,4,2_8)} \hskip.1cm , \\
&p(5n+3)_{(3,1,0,1,1_8),\hskip.1cm (2,0,0,3,1_8)} \hskip.1cm , \\
&p(5n+4)_{(4,4,3_8),\hskip.1cm (1,1,0,1,3_8),\hskip.1cm (2,3,0,1,3_8),\hskip.1cm 
(3,4,0,4,3_8),\hskip.1cm (2,4,0,1,4_8),\hskip.1cm (3,0,4,4_8)} \hskip.1cm .
\end{align*}   

\subsection{Some examples of the form $p(7n+B)_{\mathbf e}\equiv 0\pmod 7$}
\label{subsectionp=7}
\begin{align*}
&p(7n+2)_{(4)},\hskip.1cm p(7n+3)_{(6)},\hskip.1cm p(7n+4)_{(4),\hskip.1cm (6)},\hskip.1cm 
p(7n+5)_{(1),\hskip.1cm (4)},\hskip.1cm p(7n+6)_{(4),\hskip.1cm (6)},\hskip.1cm \\
&p(7n+2)_{(2,2),\hskip.1cm (1,5),\hskip.1cm (3,5)},\hskip.1cm 
p(7n+3)_{(5,1),\hskip.1cm (2,2)},\hskip.1cm 
p(7n+4)_{(1,2),\hskip.1cm (2,2),\hskip.1cm (4,4),\hskip.1cm (1,5)},\hskip.1cm \\
&p(7n+5)_{(5,1),\hskip.1cm (1,5),\hskip.1cm (5,5)},\hskip.1cm 
p(7n+6)_{(2,1),\hskip.1cm (5,1),\hskip.1cm (2,2),\hskip.1cm (5,3)},\hskip.1cm 
\end{align*}
\begin{align*}
&p(7n+2)_{(6,1,0,3),\hskip.1cm (3,5,0,3),\hskip.1cm (4,0,0,4),\hskip.1cm (1,5,0,4),\hskip.1cm 
(5,1,0,5),\hskip.1cm (6,1,0,6)},\hskip.1cm \\
&p(7n+3)_{(1,4,0,1),\hskip.1cm (2,2,0,2),\hskip.1cm (5,1,0,4),\hskip.1cm (5,1,0,5),\hskip.1cm (2,2,0,6)},\hskip.1cm \\
&p(7n+4)_{(1,4,0,1),\hskip.1cm (3,6,0,1),\hskip.1cm (3,2,0,3),\hskip.1cm 
(3,5,0,3),\hskip.1cm (4,1,0,5),\hskip.1cm (5,1,0,5),\hskip.1cm (6,1,0,6),\hskip.1cm (6,5,0,6)},\hskip.1cm \\
&p(7n+5)_{(2,2,0,2),\hskip.1cm (2,6,0,2),\hskip.1cm (4,3,0,3),\hskip.1cm 
(3,5,0,3),\hskip.1cm (3,1,0,6),\hskip.1cm (6,1,0,6),\hskip.1cm (6,3,0,6)},\hskip.1cm \\
&p(7n+6)_{(1,4,0,1),\hskip.1cm (4,5,0,1),\hskip.1cm (2,2,0,2),\hskip.1cm (6,2,0,2),\hskip.1cm 
(2,4,0,2),\hskip.1cm (3,5,0,3),\hskip.1cm (1,6,0,3),\hskip.1cm (3,3,0,4),\hskip.1cm (5,0,0,5),\hskip.1cm (5,1,0,5)}. \hskip.1cm\end{align*}

\subsection{Some examples of the form $p(11n+B)_{\mathbf e}\equiv 0\pmod{11}$}
\label{subsectionp=11}   
\begin{align*}
&p(11n+2)_{(8)},\hskip.1cm p(11n+3)_{(10)},\hskip.1cm p(11n+4)_{(8)},\hskip.1cm \\
&p(11n+5)_{(8)},\hskip.1cm p(11n+6)_{(1),\hskip.1cm (10)},\hskip.1cm \\
&p(11n+7)_{(3),\hskip.1cm (8)},\hskip.1cm 
p(11n+8)_{(5),\hskip.1cm (8),\hskip.1cm (10)},\hskip.1cm \\
&p(11n+9)_{(7),\hskip.1cm (8),\hskip.1cm (10)},\hskip.1cm p(11n+10)_{(10)}.\hskip.1cm 
\end{align*}
\begin{align*}
&p(11n+2)_{(9,1),\hskip.1cm (2,6),\hskip.1cm (1,9)},\hskip.1cm 
p(11n+3)_{(4,1),\hskip.1cm (6,2),\hskip.1cm (2,6)},\hskip.1cm 
p(11n+4)_{(2,3),\hskip.1cm (2,6)},\hskip.1cm \\
&p(11n+5)_{(6,2),\hskip.1cm (7,7),\hskip.1cm (1,9)},\hskip.1cm 
p(11n+6)_{(9,1),\hskip.1cm (6,2),\hskip.1cm (2,5),\hskip.1cm (2,6),\hskip.1cm (9,7)},\hskip.1cm \\
&p(11n+7)_{(9,1),\hskip.1cm (2,6),\hskip.1cm (1,9),\hskip.1cm (7,9)},\hskip.1cm
p(11n+8)_{(9,1),\hskip.1cm (6,2),\hskip.1cm (8,4),\hskip.1cm (1,9),\hskip.1cm (9,9)},\hskip.1cm \\
&p(11n+9)_{(3,2),\hskip.1cm (6,2),\hskip.1cm (2,6),\hskip.1cm (6,6),\hskip.1cm (1,9)},\hskip.1cm 
p(11n+10)_{(9,1),\hskip.1cm (5,2),\hskip.1cm (6,2),\hskip.1cm (1,4),\hskip.1cm (4,8)},\hskip.1cm 
\end{align*}
\begin{align*}
&p(11n+2)_{(3,2,0,2),\hskip.1cm (2,6,0,2),\hskip.1cm (6,6,0,2),\hskip.1cm (3,2,0,3)},\hskip.1cm \\
&p(11n+2)_{(5,2,0,7),\hskip.1cm (9,1,0,9),\hskip.1cm (8,0,0,10),\hskip.1cm (10,1,0,10)},\hskip.1cm \\
&p(11n+3)_{(1,8,0,1),\hskip.1cm (5,9,0,4),\hskip.1cm (5,9,0,5),\hskip.1cm (7,2,0,7),\hskip.1cm 
(9,1,0,9),\hskip.1cm (6,2,0,10)},\hskip.1cm \\
&p(11n+4)_{(8,9,0,1),\hskip.1cm (3,2,0,3),\hskip.1cm (3,0,0,4),\hskip.1cm (9,2,0,7),\hskip.1cm 
(9,1,0,9),\hskip.1cm (2,7,0,9),\hskip.1cm (9,9,0,9)}.\hskip.1cm 
\end{align*}
\begin{align*}
&p(11n+5)_{(6,0,0,1),\hskip.1cm (3,2,0,3),\hskip.1cm (10,5,0,3),\hskip.1cm (1,2,0,4)},\hskip.1cm \\
&p(11n+5)_{(4,6,0,4),\hskip.1cm (5,9,0,5),\hskip.1cm (5,7,0,6),\hskip.1cm (10,1,0,10)},\hskip.1cm \\
&p(11n+6)_{(4,2,0,1),\hskip.1cm (1,8,0,1),\hskip.1cm (2,1,0,2),\hskip.1cm (2,6,0,2),\hskip.1cm (8,7,0,3)},\hskip.1cm \\
&p(11n+6)_{(5,9,0,5),\hskip.1cm (3,9,0,6),\hskip.1cm (7,3,0,8),\hskip.1cm (9,0,0,9),\hskip.1cm 
(9,1,0,9),\hskip.1cm (10,3,0,10)}.\hskip.1cm 
\end{align*}
\begin{align*}
&p(11n+7)_{(4,1,0,2),\hskip.1cm (2,6,0,2),\hskip.1cm (3,2,0,3),\hskip.1cm (6,9,0,3),\hskip.1cm 
(4,8,0,4),\hskip.1cm (10,3,0,5),\hskip.1cm (1,0,0,6)},\hskip.1cm \\
&p(11n+7)_{(8,2,0,6),\hskip.1cm (5,5,0,8),\hskip.1cm (8,9,0,8),\hskip.1cm 
(9,1,0,9),\hskip.1cm (7,2,0,9),\hskip.1cm (3,4,0,9),\hskip.1cm (10,1,0,10)},\hskip.1cm \\
&p(11n+8)_{(1,8,0,1),\hskip.1cm (2,3,0,2),\hskip.1cm (2,6,0,2),\hskip.1cm 
(4,0,0,3),\hskip.1cm (3,2,0,3),\hskip.1cm (3,6,0,3)},\hskip.1cm \\
&p(11n+8)_{(9,1,0,4),\hskip.1cm (8,5,0,5),\hskip.1cm (5,9,0,5),\hskip.1cm 
(10,2,0,6),\hskip.1cm (6,4,0,6)},\hskip.1cm \\
&p(11n+8)_{(2,6,0,6),\hskip.1cm (1,10,0,7),\hskip.1cm (3,7,0,8),\hskip.1cm 
(7,1,0,10),\hskip.1cm (10,1,0,10)}.\hskip.1cm 
\end{align*}
\begin{align*}
&p(11n+9)_{(1,8,0,1),\hskip.1cm (9,8,0,1),\hskip.1cm (2,2,0,3),\hskip.1cm (3,2,0,3),\hskip.1cm (9,4,0,3)},\hskip.1cm \\
&p(11n+9)_{(4,10,0,4),\hskip.1cm (5,2,0,5),\hskip.1cm (6,7,0,5),\hskip.1cm (2,9,0,5),\hskip.1cm (5,9,0,5)},\hskip.1cm \\
&p(11n+9)_{(10,1,0,7),\hskip.1cm (8,0,0,8),\hskip.1cm (1,9,0,8),\hskip.1cm (9,1,0,9),\hskip.1cm 
(10,1,0,10),\hskip.1cm (5,3,0,10)},\hskip.1cm \\
&p(11n+10)_{(1,8,0,1),\hskip.1cm (7,10,0,1),\hskip.1cm (2,6,0,2),\hskip.1cm (9,7,0,2),\hskip.1cm 
(5,9,0,2),\hskip.1cm (2,1,0,4)},\hskip.1cm \\
&p(11n+10)_{(7,2,0,5),\hskip.1cm (4,9,0,5),\hskip.1cm (5,9,0,5),\hskip.1cm 
(6,6,0,6),\hskip.1cm (8,3,0,7),\hskip.1cm (7,9,0,7)},\hskip.1cm \\
&p(11n+10)_{(10,0,0,8),\hskip.1cm (6,2,0,8),\hskip.1cm (4,1,0,9),\hskip.1cm 
(1,8,0,9),\hskip.1cm (3,5,0,10),\hskip.1cm (10,7,0,10)}.\hskip.1cm \\
\end{align*}

\newpage

\subsection{Some examples of the form $p(13n+B)_{\mathbf e}\equiv 0\pmod{13}$}
\label{subsectionp=13}
An incomplete list of (conjectural) primitive 
examples modulo $13$ involving only 
unit-roots of order at most $4$ is given by:
\begin{align*}
&p(13n+2)_{(11,1),\hskip.1cm (2,8),\hskip.1cm (2,8,0,2),\hskip.1cm (8,8,0,6),\hskip.1cm 
(11,1,0,11),\hskip.1cm (5,6,0,11)},\hskip.1cm \\
&p(13n+3)_{(12),\hskip.1cm (8,2),\hskip.1cm (1,10,0,1),\hskip.1cm 
(5,0,0,5),\hskip.1cm (10,6,0,6),\hskip.1cm (3,10,0,9)},\hskip.1cm \\
&p(13n+4)_{(10),\hskip.1cm (12),\hskip.1cm (8,2),\hskip.1cm (2,8),\hskip.1cm 
(1,11),\hskip.1cm (2,6,0,1),\hskip.1cm (1,10,0,1),\hskip.1cm (3,4,0,3),\hskip.1cm \dots},\\
&p(13n+5)_{(10),\hskip.1cm (11,1),\hskip.1cm (1,11),\hskip.1cm (6,1,0,2),\hskip.1cm 
(2,8,0,2),\hskip.1cm (3,4,0,3),\hskip.1cm \dots},\\
&p(13n+6)_{(12),\hskip.1cm (11,1),\hskip.1cm (8,2),\hskip.1cm (2,8),\hskip.1cm 
(1,10,0,1),\hskip.1cm (2,8,0,2),\hskip.1cm (8,12,0,2),\hskip.1cm \dots},\\
&p(13n+7)_{(10),\hskip.1cm (11,1),\hskip.1cm (8,2),\hskip.1cm (6,3),\hskip.1cm (1,11),\hskip.1cm 
(2,8,0,2),\hskip.1cm (10,10,0,2),\hskip.1cm (3,4,0,3),\hskip.1cm \dots},\\
&p(13n+8)_{(10),\hskip.1cm (12),\hskip.1cm (8,1),\hskip.1cm (11,1),\hskip.1cm (8,2),\hskip.1cm 
(1,10,0,1),\hskip.1cm (2,8,0,2),\hskip.1cm (12,8,0,2),\hskip.1cm (8,10,0,2),\hskip.1cm \dots},\\
&p(13n+9)_{(10),\hskip.1cm (2,8),\hskip.1cm (1,11),\hskip.1cm (10,12),\hskip.1cm 
(12,9,0,1),\hskip.1cm (1,6,0,2),\hskip.1cm (10,8,0,2),\hskip.1cm \dots},\\
&p(13n+10)_{(12),\hskip.1cm (8,2),\hskip.1cm (2,8),\hskip.1cm (12,10),\hskip.1cm (8,12),\hskip.1cm 
(1,7,0,1),\hskip.1cm (1,10,0,1),\hskip.1cm (5,1,0,3),\hskip.1cm \dots},\\
&p(13n+11)_{(10),\hskip.1cm (12),\hskip.1cm (11,1),\hskip.1cm (8,2),\hskip.1cm (1,8),\hskip.1cm 
(10,10),\hskip.1cm (1,11),\hskip.1cm (3,5,0,1),\hskip.1cm (2,8,0,2),\hskip.1cm \dots},\\
&p(13n+12)_{(10),\hskip.1cm (3,6),\hskip.1cm (2,8),\hskip.1cm (12,8),\hskip.1cm (1,11),\hskip.1cm 
(5,3,0,1),\hskip.1cm (1,5,0,1),\hskip.1cm (7,0,0,2),\hskip.1cm \dots}.
\end{align*}

\subsection{Computational aspects}
\label{subsectionComput}

We outline here briefly the discovery of the (conjectural) 
generalized Ramanujan congruences previously described.

The computations where done in two steps. In a first step,
we used series expansions of $\sum_{n=0}^\infty p(n)q^n$
(with coefficients reduced modulo a small fixed prime $l$)
and its powers up to degree $N \sim 200$ in order to guess them. 
In a second step, we redid the computations up to degree
$N=5000$ for the discovered examples 
(we did not encounter false positives, 
they should be rare since the probability for a false
positive should naively be close to $l^{-N/l}$ for examples
of the kind considered here.)

Conjectural examples where guessed by considering all possible 
exponents $e_i \in \{0, \dots, p-1\}$ for $i$ ranging over the set 
$\mathcal D(a)$ of all divisors of a small integer $a$
(we considered mainly $a \in \{2,3,4,6,8\}$). 
We wrote a small Maple-program generating all
$l^{\vert \mathcal D(a)\vert}$ possible series $\prod_{i \in \mathcal D(a)} A_i^{e_i}$
up to order $N$ over $\mathbb F_l$ where $A_j = \sum_{n=0}^\infty P(n) q^{jn}$
(with coefficients reduced modulo $l$ and working only up to degree $N$),
and checking for generalized Ramanujan congruences up to order $N$.

Examples of generalized Ramanujan congruences
seem surprisingly abundant, it is not hard to find them,
they come in large numbers and seem to be very common.

\end{document}